\documentclass[a4paper,10pt]{amsart}

\usepackage{a4}
\usepackage{graphics,graphicx}
\usepackage{amssymb,amsmath}
\usepackage[all]{xy}
\usepackage{longtable}
\usepackage{multirow}
\usepackage{showlabels}
\CompileMatrices
\usepackage{hyperref}
\usepackage{skull}

\usepackage[backrefs,lite]{amsrefs}

\usepackage{tikz}
\usetikzlibrary{calc}
\usetikzlibrary{decorations.markings}
\usetikzlibrary{arrows,shapes,matrix}
\usetikzlibrary{positioning,shapes,shadows,arrows}

\usepackage{dpfloat}

\newtheorem{theorem}{Theorem}[section]
\newtheorem{lemma}[theorem]{Lemma}
\newtheorem{proposition}[theorem]{Proposition}
\newtheorem{corollary}[theorem]{Corollary}

\newtheorem{addendum}[theorem]{Addendum}

\theoremstyle{definition}

\newtheorem{definition}[theorem]{Definition}

\newtheorem{remark}[theorem]{Remark}
\newtheorem{method}[theorem]{Method}

\newtheorem{algorithm}[theorem]{Algorithm}

\theoremstyle{remark}

\def\TT{\mathbb{T}}
\def\KK{\mathbb{K}}
\def\ZZ{\mathbb{Z}}

\def\PP{\mathbb{P}}
\def\QQ{\mathbb{Q}}

\def\<{\langle}
\def\>{\rangle}

\def\cox{\mathcal{R}}

\makeatletter
\newcommand{\thickhline}{%
    \noalign {\ifnum 0=`}\fi \hrule height 1pt
    \futurelet \reserved@a \@xhline
}
\makeatother

\renewcommand{\phi}{\varphi}

\def\div{{\rm div}}

\def\bangle#1{\langle #1 \rangle}

\def\KK{{\mathbb K}}
\def\TT{{\mathbb T}}
\def\ZZ{{\mathbb Z}}

\def\QQ{{\mathbb Q}}
\def\PP{{\mathbb P}}

\def\Cox{\cox}

\def\Eff{{\rm Eff}}
\def\Mov{{\rm Mov}}

\def\Ample{{\rm Ample}}
\def\SAmple{{\rm SAmple}}
\def\Nef{{\rm Nef}}

\def\Aut{\operatorname{Aut}}

\def\Cl{\operatorname{Cl}}

\def\Aut{{\rm Aut}}
\def\Eff{{\rm Eff}}

\usepackage{lmodern}
\makeatletter
\ifcase \@ptsize \relax% 10pt
  \newcommand{\miniscule}{\@setfontsize\miniscule{3}{5}}% \tiny: 5/6
\fi
\makeatother

\makeindex

\author[J.~Hausen, S.~Keicher, A.~Laface]{J\"urgen~Hausen, Simon~Keicher and Antonio~Laface}

 \address{Mathematisches Institut, Universit\"at T\"ubingen,
Auf der Morgenstelle 10, 72076 T\"ubingen, Germany}
\email{juergen.hausen@uni-tuebingen.de}

\address{Departamento de Matem\'atica, Universidad de Concepci\'on,
Casilla 160-C, Concepci\'on, Chile}
\email{keicher@mail.mathematik.uni-tuebingen.de}

\address{Departamento de Matem\'atica, Universidad de Concepci\'on,
Casilla 160-C, Concepci\'on, Chile}
\email{alaface@udec.cl}

\title[On blowing up the weighted projective plane]%
{On blowing up the weighted projective plane} 

\thanks{The second author was supported by proyecto 
FONDECYT postdoctorado N.~3160016.
The third author was supported by proyecto 
FONDECYT regular N.~1150732 and 
grant Anillo CONICYT PIA ACT 1415.}

\begin{document}

\begin{abstract}
We investigate the blow-up of a weighted projective plane
at a general point. 
We provide criteria and algorithms for testing if 
the result is a Mori dream surface and we compute the 
Cox ring in several cases.
Moreover applications to the study of $\overline{M}_{0,n}$
are discussed.
\end{abstract}

\maketitle 

\section{Introduction}

Let $a,b,c$ be pairwise coprime positive integers
and denote by $\PP(a,b,c)$ the associated weighted 
projective plane, defined over an algebraically 
closed field $\KK$ of characteristic zero. 
We consider the blow-up 
$$
\pi \colon X(a,b,c) \ \to \ \PP(a,b,c)
$$
at the point $[1,1,1] \in \PP(a,b,c)$ and ask whether 
$X = X(a,b,c)$ 
is a Mori dream surface, i.e., has finitely 
generated Cox ring
$$ 
\mathcal{R}(X)
\ = \
\bigoplus_{\Cl(X)} \Gamma(X,\mathcal{O}(D)).
$$
This problem has been studied  by several authors 
and the results have been used  
to prove that $\overline{M}_{0,n}$ is not a 
Mori dream space for $n \ge 13$, 
see~\cite{GoNiWa,CaTe,gk}.
In fact, as we will see below,  $\overline{M}_{0,n}$
is not even a Mori dream space for $n \ge 10$.

However, it still remain widely open questions, 
which of the $X(a,b,c)$ are Mori dream surfaces, 
and, if so, how does their Cox ring look like.
We provide new results and computational tools.
Our approach goes through the description of the Cox 
ring of $X = X(a,b,c)$ as a saturated Rees algebra:
$$ 
\mathcal{R}(X)
\ = \ 
S[I]^{\rm sat} 
\ := \ 
\bigoplus_{\mu \in \ZZ} (I^{-\mu} : J^\infty)t^\mu,
$$ 
where $S$ is the Cox ring of $\PP(a,b,c)$ 
and $I,J \subseteq S$ are the 
weighted homogeneous ideals of  
the points $(1,1,1)$ and $(0,0,0)$ respectively;
see~\cite[Prop.~5.2]{hkl}.
We say that an element 
of the Cox ring $\mathcal{R}(X)$ 
is of \emph{Rees multiplicity} $\mu$
if it belongs to the component $(I^{\mu} : J^\infty)t^{-\mu}$.

Our theoretical results concern the cases 
that the Cox ring of $X$ is generated
by elements of low Rees multiplicity.
We characterize this situation in terms 
of $a,b,c$ and we provide generators 
and relations for the Cox ring of $X$, 
where we list the degree of a generator $T_i$ 
in $\Cl(X) = \ZZ^2$ as the $i$-th 
column of the degree matrix~$Q$.

\begin{theorem}
\label{isKstar}
Let $X = X(a,b,c)$ be as before.
Then the following statements are equivalent.
\begin{enumerate}
\item
The surface $X$ admits a nontrivial 
$\KK^*$-action.
\item
One of the integers $a,b,c$ lies in the monoid
generated by the other two.
\item
The Cox ring of $X$ is generated by homogeneous 
elements of Rees multiplicity at most one.
\end{enumerate}

\goodbreak

\noindent
If one of these conditions holds, then 
$X$ is a Mori dream surface.
Moreover, if $a$ lies in the monoid 
generated by $b$ and $c$, then the Cox ring 
of $X$ is given by
$$
\mathcal{R}(X)
\ = \ 
\KK[T_1, \ldots, T_5] / \bangle{T_4T_5 - T_1^c + T_2^b},
\qquad
Q
\ = \ 
\left[
\mbox{\footnotesize $
\begin{array}{rrrrr}
b & c &  a & bc & 0 
\\
0 & 0 & -1 & -1 & 1 
\end{array}
$}
\right]
$$
and the Rees multiplicities of the generators 
$T_1, \ldots, T_5$ are $0,0,1,1,-1$ respectively.
In particular, $X$ is a toric surface if and only
if at least one of the three integers $a,b,c$ equals one.
\end{theorem}

The first step beyond $\KK^*$-surfaces means 
generation of the Cox ring in Rees multiplicity 
at most two.
Our result yields in particular that in this 
case only one generator of Rees multiplicity two 
is needed.

\begin{theorem}
\label{1112}
Assume that none of $a,b,c$ is contained in the 
monoid generated by the remaining two.
Then, for $X = X(a,b,c)$, the following statements
are equivalent.
\begin{enumerate}
\item
The Cox ring of $X$ is generated
by elements of Rees multiplicity at most two.
\item
After suitably reordering $a,b,c$, one has 
$2a = nb + mc$ with positive integers $n,m$ 
such that $b \geq 3m$ and $c\geq 3n$.
\end{enumerate}
Moreover, if one of these conditions holds, 
then $X$ is a Mori dream surface 
and its Cox ring is given by 
$$
\begin{array}{rcl}
\mathcal{R}(X)
& = &
\KK[x,y,z,s_1,\ldots,s_4,t]/(I_2:t^\infty),
\\[2ex]
Q 
& = &
\left[\arraycolsep=3.2pt
    \mbox{\footnotesize$
    \begin{array}{cccccccc}
     a & b & c & 2a & \tfrac{b(c+n)}{2} & \tfrac{c(b+m)}{2} & bc & 0\\
     0 & 0 & 0 & -1 & -1 & -1 & -2 & 1\\
    \end{array}
    $}
  \right],
\end{array}
$$
where the Rees multiplicities of the generators 
$x,y,z,s_1,\ldots,s_4,t$ are $0,0,0,1,1,1,2,-1$,
respectively,
and the ideal $I_2 \subseteq \KK[x,y,z,s_1,\ldots,s_4,t]$ 
is generated by the polynomials
$$
\begin{array}{c}
x^2 - y^nz^m - s_1t, 
\qquad
xz^\frac{b-m}{2}-y^\frac{c+n}{2} - s_2t, 
\qquad
xy^\frac{c-n}{2}-z^\frac{b+m}{2}- s_3t, 
\\[1ex]
xy^\frac{c-3n}{2}z^\frac{b-3m}{2}s_1-y^\frac{c-n}{2}s_2-z^\frac{b-m}{2}s_3-s_4t,
\qquad
y^\frac{c-3n}{2}z^\frac{b-3m}{2}s_1^2 - s_2s_3-xs_4,
\\[1ex]
y^\frac{c-n}{2}s_1-z^ms_2-xs_3,
\qquad
z^\frac{b-m}{2}s_1-xs_2-y^ns_3,
\\[1ex]
s_3^2 + y^\frac{c-3n}{2}s_1s_2-z^ms_4,
\qquad
s_2^2 + z^\frac{b-3m}{2}s_1s_3-ys_4.
\end{array}
$$
\end{theorem}

In fact, we expect the ideal $I_2$ generated 
by the polynomials displayed in Theorem~\ref{1112} 
to be prime and thus to coincide 
with the saturation $I_2:t^\infty$.
As we will see in Corollary~\ref{cor:3bc},
Theorem~\ref{1112}  comprises the in particular
the surfaces $X(3,b,c)$ such that none of $3,b,c$ 
lies in the monoid generated by the remaining two.

In Section~\ref{sec:algos}, we present 
computational tools and discuss 
applications to the study of $\overline{M}_{0,n}$.
Algorithm~\ref{ismds} verifies a guess of
generators for the Cox ring of a blow-up
of an arbitrary Mori dream space.
Moreover, Algorithm~\ref{algo:ismds} implements 
the Mori dreamness criterion for $X(a,b,c)$ 
given in Proposition~\ref{prop:mdscriterion2}.
As an application, we obtain:%

\begin{theorem}
\label{thm:abc30}
Let $a<b<c\leq 30$ be pairwise coprime positive
integers. 
Then $X(a,b,c)$ is a Mori dream surface whenever 
the triple $a,b,c$ does not occur in the 
following list.

\begingroup
\allowdisplaybreaks
\footnotesize
\begin{gather*}
\begin{array}{|rrrr|}
\hline
a & b & c& \\
\hline
 7 & 10 & 19& \bullet\\
% 7 & 10 & 19& \bullet\\
 7 & 11 & 20& \\
 7 & 13 & 16& \\
 7 & 13 & 23& \\
 7 & 13 & 24& \\
 7 & 15 & 19& \\
 7 & 15 & 26& \star \\
 7 & 16 & 17& \\
 7 & 16 & 29& \\
 7 & 17 & 22& \star\\
 7 & 17 & 29& \star\\
 7 & 18 & 19& \\
 7 & 19 & 22& \bullet\\
 7 & 19 & 25& \star\\
 7 & 20 & 23& \\
 7 & 23 & 24& \\
 7 & 23 & 27& \bullet\\
 7 & 24 & 29& \\
 7 & 25 & 26& \\
 7 & 25 & 27& \\
 7 & 26 & 29& \bullet\\
 9 & 10 & 13& \\
 9 & 10 & 23& \\
 9 & 11 & 17& \\
\hline
 \end{array}
 \ 
 \begin{array}{|rrrr|}
\hline
 a & b & c& \\
 \hline
 9 & 11 & 28& \\
 9 & 13 & 16& \\
 9 & 13 & 29& \star\\
 9 & 16 & 19& \\
 9 & 17 & 23& \\
 9 & 19 & 22& \\
 9 & 22 & 25& \\
 9 & 23 & 29& \\
 9 & 25 & 28& \\
 10 & 11 & 17& \\
 10 & 11 & 27 &\star\\
 10 & 13 & 21 &\star\\
 10 & 17 & 19& \\
 10 & 17 & 29& \star\\
 10 & 19 & 23& \\
 10 & 27 & 29& \star\\
 11 & 12 & 19& \\
 11 & 13 & 16& \\
 11 & 13 & 19& \star\\
 11 & 13 & 21& \\
 11 & 13 & 27& \\
 11 & 16 & 25& \star\\
 11 & 16 & 29& \\
 11 & 17 & 24& \\
 \hline
 \end{array}
 \ 
 \begin{array}{|rrrr|}
\hline
 a & b & c& \\
 \hline
 11 & 17 & 26& \\
 11 & 17 & 30& \\
 11 & 18 & 23& \\
 11 & 19 & 29& \star\\
 11 & 21 & 25& \star\\
 11 & 21 & 29& \\
 11 & 23 & 30& \\
 12 & 13 & 17& \star\\
 12 & 17 & 25& \\
 12 & 19 & 23& \star\\
 12 & 25 & 29& \star\\
 13 & 14 & 17& \\
 13 & 14 & 23& \\
 13 & 15 & 23& \\
 13 & 16 & 17& \\
 13 & 16 & 27& \\
 13 & 17 & 23& \star\\
 13 & 17 & 27& \\
 13 & 18 & 25& \star\\
 13 & 19 & 21& \\
 13 & 19 & 24& \star\\
 13 & 19 & 28& \\
 13 & 19 & 30& \\
 13 & 22 & 23& \\
 \hline
 \end{array}
 \ 
 \begin{array}{|rrrr|}
\hline
 a & b & c& \\
 \hline
 13 & 22 & 25& \\
 13 & 22 & 29& \star\\
 13 & 23 & 27& \\
 13 & 25 & 29& \\
 13 & 29 & 30& \\
 14 & 17 & 19& \\
 14 & 17 & 29& \\
 14 & 19 & 27& \star\\
 14 & 23 & 25& \\
 14 & 23 & 29& \\
 14 & 25 & 29& \star\\
 16 & 17 & 21& \\
 16 & 17 & 23& \\
 16 & 17 & 27& \star\\
 16 & 19 & 21& \\
 16 & 21 & 25& \\
 16 & 25 & 29& \\
 16 & 27 & 29& \\
 17 & 19 & 22& \\
 17 & 19 & 26& \\
 17 & 19 & 27& \\
 17 & 20 & 21& \star\\
 17 & 21 & 22& \star\\
 17 & 22 & 25& \\
\hline
 \end{array}
 \ 
 \begin{array}{|rrrr|}
\hline
 a & b & c& \\
 \hline
 17 & 23 & 24& \star\\
 17 & 23 & 25& \star\\
 17 & 23 & 27& \\
 17 & 23 & 30& \\
 17 & 25 & 27& \star\\
 17 & 26 & 29& \star\\
 17 & 27 & 28& \star\\
 17 & 29 & 30& \star\\
 18 & 19 & 23& \\
 18 & 23 & 25& \star\\
 19 & 20 & 27& \star\\
 19 & 21 & 29& \star\\
 19 & 23 & 29& \\
 19 & 25 & 29& \\
 19 & 26 & 27& \\
 19 & 26 & 29& \star\\
 21 & 23 & 26& \\
 21 & 25 & 26& \\
 22 & 23 & 27& \star\\
 22 & 25 & 27& \\
 23 & 25 & 28& \star\\
 23 & 27 & 28& \\
 23 & 29 & 30& \\
 &&&\\
 \hline
\end{array}
\end{gather*}
\endgroup
The triples $a,b,c$ marked with $\star$ are known 
to give non Mori dream surfaces, see~\cite{gk}.
For the other listed $a,b,c$, the Cox ring 
of $X(a,b,c)$ needs generators of Rees multiplicities 
at least 15.
\end{theorem}

The fact that  all $X(a,b,c)$ with $\min(a,b,c) \le 6$ 
are Mori dream surfaces is due to Cutkosky~\cite{ck}.
Besides the cases covered by Theorems~\ref{isKstar} 
and~\ref{1112}, Theorem~\ref{thm:abc30} yields 
514 new Mori dream surfaces $X(a,b,c)$.
The question whether or not the~$X(a,b,c)$ listed without 
$\star$ in Theorem~\ref{thm:abc30} are Mori dream surfaces 
remains open --- in fact, we expect some of them to be 
Mori dream surfaces, e.g. those marked with~$\bullet$.%

Let us discuss the applications to the question 
whether or not $\overline M_{0,n}$ is a Mori dream space.
Recall that for $n \le 6$, there is an affirmative 
answer~\cite{Ca}.
For higher~$n$, the idea of Castravet and Tevelev~\cite{CaTe} 
is to construct sequences 
$$
\xymatrix{
{\overline M_{0,n}}
\ar[r]
&
{\overline L}_n' 
\ar@{-->}[r]
&
X(a,b,c),
}
$$
where the first arrow is the canonical proper surjections 
onto the blow-up ${\overline L}_n'$ of the Losev-Manin 
space ${\overline L}_n$ at the general point and the 
second one is a composition of small quasimodifications
and proper surjections.
This allows to conclude that if $X(a,b,c)$ is 
not a Mori dream space, the same holds for 
$\overline M_{0,n}$.
Applying results from~\cite{GoNiWa}, 
Castravet and Tevelev obtain that $\overline M_{0,n}$ 
is not a Mori dream space for $n \ge 134$.
Gonzales and Karu~\cite{gk} gave further 
sufficient conditions on $X(a,b,c)$ 
to be not a Mori dream surface and, as a consequence,
showed that $\overline M_{0,n}$ is not a Mori dream 
space for $n \ge 13$.
In fact, as we will see, the results of~\cite{gk} 
even lead to the following:

\begin{addendum}
\label{thm:add}
$\overline M_{0,n}$ is not a Mori dream space for $n \ge 10$.
\end{addendum}

For the remaining open cases $n = 7,8,9$, our 
algorithms yield that all $X(a,b,c)$ that can 
be reached via a surjection of any modified 
Losev-Manin space $L_n''$ as in the above sequence
are Mori dream surfaces. 
In particular, the treatment of the cases 
$n = 7,8,9$ needs new ideas.

\tableofcontents

%%%%%%%%%%%%%%%%%%%%%%%%%%%%%%
\section{Orthogonal pairs I}
\label{sec:orthpoly}

Here we introduce our main tool to decide when 
a given $X = X(a,b,c)$ is a Mori dream surface.
It depends on the specific situation and it allows
to answer the question entirely in terms of 
(computable) data of $\PP(a,b,c)$, see 
Proposition~\ref{prop:mdscriterion2}.
We first introduce the necessary notation and 
recall some background.

Let pairwise coprime positive integers $a,b,c$
be given.
The homogeneous coordinate ring of the weighted
projective plane $\PP(a,b,c)$ is the $\ZZ$-graded 
polynomial ring
$$
S \ := \ \KK[x,y,z],
\qquad
\deg(x) \ := \ a,
\quad 
\deg(y) \ := \ b,
\quad 
\deg(z) \ := \ c.
$$
For a homogeneous polynomial $f \in S_d$, we denote 
by $V(f)$ the associated (not necessarily reduced)
curve on $\PP(a,b,c)$.
The divisor class group of $\PP(a,b,c)$ is freely 
generated by 
$$
A \ := \ \eta V(x) + \zeta V(y),
$$
where we fix $\eta, \zeta \in \ZZ$ with 
$\eta a + \zeta b = 1$.
We regard the Cox ring of $\PP(a,b,c)$ as a 
divisorial algebra
$$ 
\mathcal{R}(\PP(a,b,c))
\ = \ 
\bigoplus_{d \in \ZZ} \Gamma(\PP(a,b,c), \mathcal{O}(d A)).
$$
Observe that the identification of this algebra with 
the homogeneous coordinate ring $S$ goes via
$$ 
S_d \ni f 
\ \mapsto \ 
fx^{-d\eta}y^{-d\zeta} \in \Gamma(\PP(a,b,c), \mathcal{O}(d A)).
$$ 

As before, $X = X(a,b,c)$ is the blow-up of $\PP(a,b,c)$ 
at the point $\mathbf{1} = [1,1,1]$ and the blow-up morphism
is denoted by $\pi \colon X  \to  \PP(a,b,c)$.
The divisor class group $\Cl(X) = \ZZ^2$
is generated by the classes of 
$$
H \, := \, \pi^*(A),
\qquad\qquad
E \, := \, \pi^{-1}(\mathbf{1}).
$$ 
In particular, the intersection form 
on $\Cl_\QQ(X)$ is determined by the 
intersection numbers 
$$
H^2 \, = \,  \frac{1}{abc},
\qquad
H \cdot E \, = \, 0,
\qquad
E^2 \, = \, -1.
$$
As we did with $\PP(a,b,c)$, we regard 
the Cox ring of $X = X(a,b,c)$ as a divisorial 
algebra. More explicitly, we write 
$$ 
\mathcal{R}(X) 
\ = \ 
\bigoplus_{(d,\mu) \in \ZZ^2} \mathcal{R}(X)_{dH + \mu E},
\qquad
\mathcal{R}(X)_{dH + \mu E} 
\ = \ 
\Gamma(X,\mathcal{O}(dH + \mu E)).
$$
The canonical pullback homomorphism $\pi^*$ realizes
the Cox ring of $\PP(a,b,c)$ as the Veronese subalgebra 
of $\ZZ H \subseteq \Cl(X)$ inside the Cox ring of $X$.
We will make use of the fact that,
as any Cox ring with torsion free grading group,
$\mathcal{R}(X)$ is a unique factorization domain.

Let $I \subseteq S$ and $J \subseteq S$ denote the 
homogeneous ideals of the points $(1,1,1) \in \KK^3$ 
and $(0,0,0) \in \KK^3$, respectively. 
Then we have the saturated Rees algebra, graded by 
$\ZZ^2$, as follows
$$ 
S[I]^{\rm sat} 
\, := \, 
\bigoplus_{\mu \in \ZZ} (I^{-\mu} : J^\infty)t^\mu
\, = \,
\bigoplus_{(d,\mu) \in \ZZ^{2}} (I^{-\mu} : J^\infty)_d t^\mu.
$$
For $f \in S[I]^{\rm sat}_{(d,\mu)}$, we refer to~$d$ as 
its \emph{degree} and to~ $-\mu$ as its 
\emph{Rees multiplicity}.
We identify the saturated Rees algebra
with the Cox ring $\mathcal{R}(X)$ 
of $X = X(a,b,c)$ via the explicit isomorphism 
$$ 
S[I]^{\rm sat}_{(d,\mu)} \ni f t^{\mu}
\ \mapsto \ 
\pi^* f  \in \mathcal{R}(X)_{dH+\mu E},
$$
see~\cite[Prop.~5.2]{hkl}.
Observe that $t \in S[I]^{\rm sat}_{(0,1)}$
is of Rees multiplicity $-1$ and,
in the Cox ring $\mathcal{R}(X)$,
it represents the canonical section of the 
exceptional divisor~$E$.
Moreover, in terms of $S$ and $S[I]^{\rm sat}$, the 
pullback map $\pi^*$ between the Cox rings of 
$\PP(a,b,c)$ and $X$ is given as
$$ 
S_d \ni f \ \mapsto \ f t^{0} \in S[I]^{\rm sat}_{(d,0)}
$$

We now assign also to every homogeneous polynomial
$f \in S_d \subseteq S$ a Rees multiplicity.

\begin{definition}
Consider a polynomial $f \in S_d \subseteq S$.
The \emph{Rees multiplicity} of $f$ is 
the maximal non-negative integer $\mu$ 
such that $f \in I^\mu:J^\infty$ holds. 
\end{definition}

\goodbreak

\begin{remark}
\label{rem:reesmult}
For every $f \in S_d \subseteq S$ and every 
$\mu \in \ZZ_{\ge 0}$,
the following statements are equivalent.
\begin{enumerate}
\item
The polynomial $f$ is of Rees multiplicity $\mu$.
\item 
The curve $V(f) \subseteq \PP(a,b,c)$ 
has multiplicity $\mu$ at 
$\mathbf{1} \in \PP(a,b,c)$.
\item
The exceptional divisor $E$ occurs with 
multiplicity $\mu$ in $\div(\pi^* f)$. 
\end{enumerate}
If $f \in S_d \subseteq S$ is of Rees multiplicity 
$\mu \in \ZZ_{\ge 0}$, then the strict transform 
of the curve~$V(f)$ in $\PP(a,b,c)$ associated 
with $f$ is given as
$$ 
\div_{dH - \mu E}(\pi^* f)
\ = \ 
\div(\pi^* f) + dH - \mu E .
$$
In particular, the element 
$f t^{-\mu} \in S[I]^{\rm sat}_{(d,-\mu)}$ 
is prime if and only if $V(f)$ is a
reduced irreducible curve, or equivalently
$f \in S$ is irreducible, 
see~\cite[Prop.~1.5.3.5]{ArDeHaLa}.
\end{remark}

\begin{definition}
\label{ort}
Let $f_1 \in S_{d_1}$ and $f_2 \in S_{d_2}$ be two 
non-constant homogeneous polynomials in~$S$ of 
Rees multiplicities $\mu_1$ and $\mu_2$ respectively.
We call $f_1, f_2$ an \emph{orthogonal pair} if 
the following holds:
\begin{enumerate}
\item
we have $d_1^2 \leq \mu_1^2abc$ and there is 
no $f_1' \in S_{d_1'}$ with $d_1' < d_1$ satisfying 
this condition; 
\item
we have $d_1d_2 = \mu_1\mu_2 abc$ and 
$f_1 \nmid f_2$ 
and there is no $f_2' \in S_{d_2'}$ with $d_2' < d_2$ 
satisfying these conditions.
\end{enumerate}
\end{definition}

\begin{proposition}
\label{prop:mdscriterion2}
As before, let $a,b,c$ be pairwise coprime positive 
integers.
Then the following statements are equivalent.
\begin{enumerate}
\item
$X(a,b,c)$ is a Mori dream surface.
\item
There exists an orthogonal pair $f_1,f_2 \in S$.
\end{enumerate}
Moreover, if~(ii) holds, then the two
polynomials $f_1,f_2 \in S$ are both  
irreducible.
\end{proposition}

\begin{proof}
In $\Cl_\QQ(X) = \QQ^2$, we consider the 
inclusions of the (two-dimensional) cones 
of ample, semiample, movable, nef and effective  
divisor classes:
$$
\Ample(X) 
\ \subseteq \ 
\SAmple(X) 
\ \subseteq \ 
\Mov(X)
\ \subseteq \ 
\Nef(X) 
\ \subseteq \
\Eff(X).
$$
The ample cone is the relative interior of 
the nef cone.
As~$H$ is semiample but not ample, it
generates an extremal ray of the semiample 
cone and thus also of the nef cone.
Moreover, the nef cone and the effective 
cone are dual to each other with respect 
to the intersection product.
In particular, $E$ generates an extremal ray of 
the effective cone because we have $H \cdot E = 0$.
Finally, from~\cite{HuKe} we know that~$X$ is a 
Mori dream surface if and only if the semiample 
cone equals the nef cone and is polyhedral in
$\Cl_\QQ(X)$.

We prove ``(i)$\Rightarrow$(ii)''. 
Since $X$ is a Mori dream surface,
the effective and the semiample cone are polyhedral
and the semiample cone equals the moving cone.
Consequently, we find non-associated prime elements
$$ 
g_1  \in  \mathcal{R}(X)_{C}, \quad C = d_1H + \mu_1 E,
\qquad\qquad
g_2  \in  \mathcal{R}(X)_{D}, \quad D = d_2H + \mu_2 E
$$
such that $C$ and $E$ generate the effective cone 
and $D$ and $H$ the semiample cone;
see~\cite[Prop.~3.3.2.1 and~Prop.~3.3.2.3]{ArDeHaLa}.
Observe that we have $\mu_i < 0 < d_i$ in 
both cases.
We choose the $g_i$ such that the degrees~$d_i$ 
are minimal with respect to the above properties.
Since the semiample cone also equals the nef cone, 
$C \cdot D = 0$ holds.

Consider the element $f_it^{\mu_i} \in S[I]^{\rm sat}_{(d_i,\mu_i)}$ 
corresponding to $g_i  \in  \mathcal{R}(X)$.
We have $f_i \in S_{d_i}$.
Moreover, we claim that $-\mu_i$ is 
the Rees multiplicity of $f_i$.
Indeed, the order of $f_i$ along $E_i$ is at least $-\mu_i$.
If it were bigger, then $f_it^{\mu_i}$ were divisible by~$t$,
which is impossible by primality of $g_i$.
Thus, Remark~\ref{rem:reesmult} gives the claim.

We check that $f_1 \in S_{d_1}$ and $f_2 \in S_{d_2}$ 
form the desired orthogonal pair.
The inequality in~\ref{ort}~(i) is due to  $C^2 \le 0$, 
the equation in~\ref{ort}~(ii) follows from 
$C \cdot D = 0$.
We verify the minimality condition for $d_1$.
Let $f \in S_d$ be of Rees multiplicity 
$\mu$ and satisfy the inequality of~(i).
Let $g  \in  \mathcal{R}(X)_{F}$,
where $F = d H -\mu E$, be the element 
corresponding to $ft^{-\mu} \in S[I]^{\rm sat}_{(d,-\mu)}$.
Then $F^2 \le 0$ holds. 
Consider 
$$ 
C_0 \ := \ \div(g_1) + C \ = \ \div(g_1) + d_1 H - \mu_1 E,
$$
$$
F_0 \ := \ \div(g) + F \ = \ \div(g) + d H - \mu E.
$$
Since $g_1 \in \mathcal{R}(X)$ is prime, 
$C_0$ is a reduced irreducible curve.
Moreover, $F_0$ is an effective curve.
The class of $C_0$ equals that of $C$ and the class of 
$F_0$ equals that of~$F$.
In particular, we have
$$ 
C_0^2 \ \le \ 0, 
\qquad
F_0^2 \ \le \ 0,
\qquad
C_0 \cdot F_0 \ \le \ 0.
$$
If $C_0^2 = 0$ holds, then all the above intersection
numbers vanish, $F$ lies on the ray through $C$ and
by the choice of $G_i$, we have $d_1 \le d$.  
If $C_0^2 < 0$ holds, then $C_0 \cdot F_0 < 0$ 
holds and we conclude that $C_0$ is a component of $F_0$.
This implies $d_1 \le d$ and we obtained the minimality 
condition for $d_1$.

We turn to the minimality condition of $d_2$.
Let $f \in S_d$ be of Rees multiplicity 
$\mu$ such that $f_1$ does not divide $f$ in $S$  
and $f$ satisfies the equation of~(ii).
As before, consider the element 
$g  \in  \mathcal{R}(X)_{F}$
corresponding to $ft^{-\mu} \in S[I]^{\rm sat}_{(d,-\mu)}$,
where $F = d H -\mu E$.
Then $F \cdot C = 0$ holds and thus
$F$ defines a class on the ray through $D$.
By the choice of $f_2$, this implies $d_2 \le d$.

We prove ``(ii)$\Rightarrow$(i)''.
Let $f_1, f_2$ form an orthogonal pair,
denote by $d_1, d_2$ the respective degrees
and by $\mu_1, \mu_2$ the Rees multiplicities.
Consider $C = d_1H-\mu_1 E$ and 
$D = d_2H-\mu_2 E$
and the elements $g_1 \in  \mathcal{R}(X)_C$ 
and $g_2 \in  \mathcal{R}(X)_D$  
corresponding to 
$f_1t^{-\mu_1} \in S[I]^{\rm sat}_{(d_1,-\mu_1)}$ and 
$f_2t^{-\mu_2} \in S[I]^{\rm sat}_{(d_2,-\mu_2)}$ 
respectively. 
By the definition of an orthogonal pair we
have $C^2 \le 0$ and $C \cdot D=0$.

We show that $g_1 \in \mathcal{R}(X)$ and 
$f_1 \in S$ are prime elements.
Otherwise, we have a desomposition 
$g_1 = g_1'h$ with homogeneous non-units 
$g_1',h \in \mathcal{R}(X)$.
Because of the minimality of $d_1$ with respect to 
$C^2 \le 0$, the corresponding decomposition 
of the degree $(d_1,-\mu_1)$ of $g_1$ is of the
shape
$$
(d_1,-\mu_1) 
\ = \ 
(d_1,-\mu_1') + (0,k) 
\ \in \
\ZZ^2
\ = \ 
\Cl(X),
$$
where $\mu_1' > \mu_1$ and $k > 0$.
We conclude that $h$ is a power of $t$, 
the canonical section of the exceptional 
divisor. 
This contradicts the fact that $\mu_1$ is 
the Rees multiplicity of $f_1$;
see Remark~\ref{rem:reesmult} 
and~\cite[Prop.~1.5.3.5]{ArDeHaLa}.
Thus, $g_1 \in \mathcal{R}(X)$ is prime,
and, again by Remark~\ref{rem:reesmult}, 
the polynomial $f_1 \in S$ is prime.

We claim that $C$ generates an extremal ray of the 
effective cone of $X$.
Otherwise, we find a prime element $g \in \mathcal{R}(X)$
such that its degree $F = dH -\mu E$, where 
$d, \mu \in \ZZ_{\ge 0}$ lies outside the 
cone generated by $E$ and $C$.
Similarly as earlier, we consider
$$ 
C_0 \ := \ \div(g_1) + C \ = \ \div(g_1) + d_1 H - \mu_1 E,
$$
$$
F_0 \ := \ \div(g) + F \ = \ \div(g) + d H - \mu E.
$$
Since $g_1$ and $g$ are prime elements in $\mathcal{R}(X)$,
these are reduced irreducible curves on $X$. 
The class of $C_0$ equals that of $C$ and the class of 
$F_0$ equals that of $F$.
In particular, we have
$$ 
C_0^2 \ \le \ 0, 
\qquad
F_0^2 \ < \ 0,
\qquad
C_0 \cdot F_0 \ < \ 0.
$$
We conclude that $F_0$ is a component of $C_0$ 
and thus $F_0 = C_0$ holds.
In particular, the class $F$ lies in the cone 
generated by $E$ and $C$; a contradiction.
We obtained that $E$ and $C$ generate the 
effective cone of $X$.%

Since $D$ is orthogonal to $C$, it generates 
an extremal ray of the nef cone of $X$. 
Thus, the nef cone of $X$ is the polyhedral 
cone generated by $D$ and $H$.
Since~$f_1$ does not divide $f_2$ in $S$, 
we conclude via Remark~\ref{rem:reesmult} 
and~\cite[Prop.~1.5.3.5]{ArDeHaLa}
that $g_1$ does not divide $g_2$ in $\mathcal{R}(X)$
and thus the curve $C_0$ is not a component of 
the effective curve $D_0 := D + \div(g_2)$.
This, together with
the fact that $nD$ is linearly equivalent 
to $rC + B$ for some $n,r \in \ZZ_{\ge 0}$ 
and a very ample divisor $B$, 
implies that the stable base 
locus of $D$ is at most zero-dimensional.
By Zariski's theorem~\cite[Theorem~6.2]{za},
one concludes that $D$ is semiample.
So, the nef cone equals the semiample 
cone and thus $X$ is a Mori dream surface.%

We turn to the supplement.
Let $f_i \in S$ and $g_i \in \mathcal{R}(X)$ 
be as in the proof of the  
implication ``(ii)$\Rightarrow$(i)''.
We already saw that $f_1$ is irreducible 
in~$S$.
To obtain irreducibility of~$f_2$ note that 
by~\cite[Prop.~3.3.2.3]{ArDeHaLa} there is 
at least one prime generator $g \in \mathcal{R}(X)$
which is not divisible by $g_1$ and has its 
degree on the ray through $D$ bounding the 
semiample cone.
The minimality condition of~\ref{ort}~(ii)
yields that $g_2$ is among these~$g$ and thus
prime.
\end{proof}

\begin{remark}
\label{rem:orthpairdegs}
Let $f_1 \in S_{d_1}$ and $f_2 \in S_{d_2}$ 
be two homogeneous polynomials in~$S$ of 
Rees multiplicities $\mu_1$ and $\mu_2$,
respectively, and assume that $f_1, f_2$ 
is an orthogonal pair.
From Proposition~\ref{prop:mdscriterion2}
and its proof, we infer the following:
\begin{enumerate}
\item
The effective cone of $X$ is polyhedral in 
$\Cl_\QQ(X)$; one ray is generated by $E$,
the other we denote by $\varrho$.
\item
The element $(d_1,-\mu_1) \in \ZZ^2 = \Cl(X)$ 
is the class of a prime divisor $C_1$
and it is the shortest non-zero lattice vector
which lies on $\varrho$ and belongs to the 
monoid of effective divisor classes of $X$.
\item
The semiample cone of $X$ is polyhedral in 
$\Cl_\QQ(X)$; one ray is generated by $H$,
the other we denote by $\tau$; here 
$\varrho = \tau$ is possible.
\item
The element $(d_2,-\mu_2) \in \ZZ^2 = \Cl(X)$
is the class of a prime divisor $C_2 \ne C_1$, 
and it is the shortest non-zero lattice 
vector which lies on $\tau$ and is the class 
of a prime divisor $C_2 \ne C_1$.
\end{enumerate}
This means in particular that for any two 
orthogonal pairs $f_1,f_2$ and $f_1',f_2'$, 
we have $d_1 = d_1'$ and
$d_2 = d_2'$ for the respective degrees
and $\mu_1 = \mu_1'$ and $\mu_2 = \mu_2'$ 
for the Rees multiplicities.
Moreover, $(d_1,-\mu_1)$ and $(d_2,-\mu_2)$ 
occur in the set of $\Cl(X)$-degrees of 
any system of homogeneous generators of 
the Cox ring $\mathcal{R}(X)$.
\end{remark}

\goodbreak

\section{Proof of Theorem~\ref{isKstar}}
\label{sec:proofs}

The setting and the notation are the same 
as in the preceding section.
We begin with preparing the proof of Theorem~\ref{isKstar}.

\begin{lemma}
\label{twocurves}
For $i=1,2$ let $f_i \in S_{d_i}$ be irreducible
of Rees multiplicity $\mu_i$
and write $C_i \subseteq X$ for the strict transform 
of $V(f_i) \subseteq \PP(a,b,c)$.
If $C_1 \cdot C_2 = 0$ holds, 
then $V(f_1) \cap V(f_2)$ contains only the
point $\mathbf{1} \in \PP(a,b,c)$.
\end{lemma}

\begin{proof}
We have $C_i = \div(\pi^* f_i) + d_iH-m_iE$.
Thus $C_1 \cdot C_2 = 0$ is equivalent to
$d_1d_2 = abc\mu_1\mu_2$. 
Bezout's theorem in $\PP(a,b,c)$ tells us 
that the zero-dimensional scheme $V(f_1,f_2)$ 
has degree $\mu_1\mu_2$.
Since $V(f_1)$ and $V(f_2)$ intersect with 
multiplicity at least $\mu_1\mu_2$ at $\mathbf{1}$,
we conclude that they intersect exactly with
multiplicity $\mu_1\mu_2$ at $\mathbf{1}$
and nowhere else.
\end{proof}

\begin{lemma}
\label{lem:orthpairdeg1}
Consider homogeneous polynomials $f_1 \in S_{d_1}$ 
and $f_2 \in S_{d_2}$, both of Rees multiplicity one,
and assume that $f_1,f_2$ is an orthogonal pair.
Then there is an orthogonal pair $f_1',f_2'$ of 
binomials $f_i' \in S_{d_i}$ of Rees multiplicity 
one.
\end{lemma}

\begin{proof}
Since $f_i$ is of Rees multiplicity 
one, we have $\mathbf{1} \in V(f_i)$. 
In particular, there are at least two 
monomials of degree $d_i$ occurring with 
non-zero coefficients in $f_i$.
We consider binomials $f_i'$ which 
are the difference of two monomials 
of $f_i$. 
Each such~$f_i'$ is of degree $d_i$.
Moreover, $V(f_i')$ has multiplicity one 
at $\mathbf{1} \in \PP(a,b,c)$ and thus
Remark~\ref{rem:reesmult} tells us that 
$f_i'$ is of Rees multiplicity one.
Observe that all binomials~$f_1'$ are
prime due to the minimality condition 
on the degree $d_1$.
Every pair $f_1',f_2'$ fulfills obviously all 
conditions of an orthogonal pair, except
$f_1' \nmid f_2'$.
In fact, this condition needs not be 
satisfied automatically.
We show how to achieve it.

If $\dim_\KK (S_{d_1}) > 2$ holds, then 
we have at least three different choices
for the binomial $f_1'$.
As the binomial $f_2'$ has at most one 
prime factor vanishing at the point~$(1,1,1)$, 
we find a pair $f_1',f_2'$ 
with $f_1' \nmid f_2'$.
We treat the case $\dim_\KK (S_{d_1}) = 2$.
Then $f_1'$ is a scalar multiple of $f_1$. 
Consider the list $f_{2,1}', \ldots, f_{2,r}'$ 
of all possible binomials made from monomials
of $f_2$.
Because of $f_2(1,1,1) = 0$, the coefficients 
of the monomials of $f_2$ sum up to zero and 
thus $f_2$ is a linear combination over the 
binomials $f_{2,j}'$.
Since $f_1$ does not divide $f_2$, there must 
be a binomial $f_2' = f_{2,j}'$ which is not
divisible by $f_1'$.
\end{proof}

\begin{lemma}
\label{reesdeg2}
Let $f_1 \in S_{d_1}$ and $f_2 \in S_{d_2}$ be 
binomials of Rees multiplicity one.
If $f_1,f_2$ is an orthogonal pair, then 
one of the numbers $a,b,c$ lies in the monoid 
generated by the remaining two.
\end{lemma}

\begin{proof}
Proposition~\ref{prop:mdscriterion2} 
tells us that $f_1$ and 
$f_2$ are both irreducible.
According to Lemma~\ref{twocurves},
the zero loci of $f_1$ and $f_2$ intersect 
only at the point $\mathbf{1} \in \PP(a,b,c)$.
Thus, reordering $a,b,c$ suitably, 
we may assume
$$
f_1 \ = \ x^{p_1}-y^{p_2}z^{p_3},
\qquad\qquad
f_2 \ = \ y^{q_1}-x^{q_2}z^{q_3}.
$$
The homogeneity of the two binomials 
and the orthogonality condition
give us the following equations: 
$$
ap_1 \ = \ bp_2 + cp_3,
\qquad
bq_1 \ = \ aq_2 + cq_3,
\qquad
p_1q_1 \ = \ c.
$$
Substituting $c = p_1q_1$ in the 
first equation and using the coprimality
of $b$ and $c$ we obtain $p_2 = p_1p_2'$ 
with a $p_2'\in \ZZ_{\geq 1}$.
Similarly one shows that $q_2 = q_1q_2'$ 
with a $q_2'\in \ZZ_{\geq 1}$.
Consider the case $p_2q_2\neq 0$.
Then, from the first two equations, we 
deduce
$$
a \ = \ bp_2' + q_1p_3,
\qquad
b \ = \ aq_2' + p_1q_3.
$$
In particular $a\geq b\geq a$, so that $a=b$,
and thus $a$ is in the monoid generated 
by $b$ and $c$.
We now treat the case $p_2q_2 = 0$. 
We may assume $q_2=0$. 
Then from $bq_1 = cq_3$ and $p_1q_1 =c$
we deduce $b = p_1q_3$. 
From the coprimality
of $b$ and $c$ we deduce $p_1 = 1$ so that
$a$ lies in the monoid generated by 
$b$ and $c$.
\end{proof}

\begin{proof}[Proof of Theorem~\ref{isKstar}]
We prove ``(i)$\Rightarrow$(ii)''.
If $X$ has a non-trivial $\KK^*$-action,
then this action stabilizes the exceptional 
curve $E \subseteq X$ and thus $\PP(a,b,c)$ 
inherits a non-trivial $\KK^*$-action having 
$[1,1,1]$ as a fixed point.
According to~\cite{Co}, this means that 
$\Aut(\PP(a,b,c))$ must contain a root 
subgroup, i.e., there must by a monomial 
in two variables in $\KK[x,y,z]$ of 
degree $a$, $b$, or $c$.
This is only possible, if one of $a,b,c$ 
lies in the monoid generated by the 
remaining two.

We show that~(ii) implies~(i), (iii) and the supplement.
We may assume that $a = mb + nc$ holds with 
non-negative integers $m$ and~$n$.
Then the morphism
$$ 
\varphi \colon 
\PP(a,b,c) 
\ \to \ 
\PP(a,b,c),
\qquad\qquad
[z_1,z_2,z_3] 
\ \mapsto \ \  
[z_1- z_2^m z_3^n,z_2,z_3]
$$
sends $[1,1,1]$ to $[0,1,1]$. 
The blowing-up of $\PP(a,b,c)$ at 
$[0,1,1]$ obviously admits a 
$\KK^*$-action. 
To obtain the Cox ring, observe 
first 
$$
\PP(a,b,c)
\ \cong \ 
V(T_4 - T_1^c + T_2^b)
\ \subseteq \ 
\PP(b,c,a,bc).
$$
The Cox ring of $X(a,b,c)$ is now computed
via a toric ambient modification, 
see~\cite[Sec.~4.1.3]{ArDeHaLa}:
We blow up $\PP(b,c,a,bc)$ at $[1,1,0,0]$.
Then $X = X(a,b,c)$ is isomorphic to the strict 
transform $V(T_4 - T_1^c + T_2^b)$ and 
its Cox ring is as claimed.
Observe that the degree matrix $Q$ is 
given with respect to the basis $H,E$ 
of $\Cl(X) = \ZZ^2$.
The last column in the degree matrix is the 
class of $E$ and thus we see that the Rees 
multiplicities of the generators are as in 
the assertion.
In particular, we obtain~(iii).

We prove ``(iii)$\Rightarrow$(ii)''.
By assumption, $X$ is a Mori dream surface.
Take homogeneous non-associated prime generators 
$g_1 \in \mathcal{R}(X)_C$ and 
$g_2 \in \mathcal{R}(X)_D$ as in 
the proof of ``(i)$ \Rightarrow$(ii)''
of Proposition~\ref{prop:mdscriterion2}.
Then the effective cone of $X$ is 
generated by $C$ and $E$ and the 
semiample cone by $D$ and $H$. 
Moreover, $g_1$ and $g_2$ occur (up to scalars) 
in any system of homogeneous generators 
of $\mathcal{R}(X)$.
Thus, since $g_1$ and $g_2$ are of positive 
Rees multiplicity, the assumption says that 
they are of Rees multiplicity one.
Let $f_i \in S_{d_i}$ denote the polynomial
such that $g_i$ corresponds to  
$\pi^*(f)t^{-1} \in S[I]^{\rm sat}_{(d_i,-1)}$.
By primality of the $g_i$, the $f_i$ are 
of Rees multiplicity one.  
Moreover, they are non-associated primes 
forming an orthogonal pair, which means 
in particular $d_1d_2 = abc$.
According to Lemma~\ref{lem:orthpairdeg1},
we may assume that 
$f_1, f_2 \in I$ are binomials.
Then Lemma~\ref{reesdeg2} gives condition~(ii).
\end{proof}

\begin{remark}
Assume that we have $c = ma + nb$ with 
non-negative integers $m$ and~$n$.
Then the describing matrix $P$ of $X(a,b,c)$ 
in the sense of \cite{ArDeHaLa} is of the 
form 
$$ 
P
\ = \ 
\left[
\mbox{\footnotesize$
\begin{array}{rrrrr}
-c & b & 0 & 0 & 0
\\
-c & 0 & 1 & 1 & 0
\\
-m & -n & 0 & 1 & 1
\end{array}$}
\right].
$$
\end{remark}

\section{Orthogonal pairs II}
\label{sec:proofs2}

The setting and the notation are as in 
the preceding sections. 
The main result is Proposition~\ref{lem:reesdeg2nottwice},
which says that if in an orthogonal pair
$f_1,f_2$ one member is of Rees multiplicity two,
then the other is not.
We will often have to compute, more or less 
explicitly, the multiplicity of a curve 
in $\PP(a,b,c)$  at the point 
$\mathbf{1} \in \PP(a,b,c)$.
For this we use the following.

\begin{remark}
\label{rem:toruscoord}
Let $0 \ne f \in S = \KK[x,y,z]$. We compute 
the mutiplicity of $V(f)$ at the point
$\mathbf{1} \in \PP(a,b,c)$. 
Consider the presentation of $\PP(a,b,c)$ 
as a quotient of $\KK^3 \setminus \{0\}$ 
by the action of $\KK^*$ given as 
$t \cdot (x,y,z) = (t^ax,t^by,t^cz)$:
$$ 
\xymatrix{
{\TT^3}
\ar@{}[r]|{\subseteq\quad}
\ar[d]_\kappa
&
{\KK^3 \setminus \{0\}}
\ar[d]^\kappa
\\
{\TT^2}
\ar@{}[r]|{\subseteq\quad}
&
{\PP(a,b,c)}
}
$$
where $\TT^k = (\KK^*)^k$ denotes the standard $k$-torus.
The restriction $\kappa \colon \TT^3 \to \TT^2$ 
of the quotient map is a homomorphism of tori and 
thus given by monomials.
Let $f_0$ be any monomial of $f$. Then we have
$$ 
\frac{f}{f_0} 
\ =  \ 
\kappa^*(h)
$$
with a unique $h \in \KK[u^{\pm 1},v^{\pm 1}]$ on $\TT^2$. 
The Laurent polynomial $h$ generates the defining ideal 
of $V(f)$ on $\TT^2$. Thus, the multiplicity of 
$V(f)$ at $\mathbf{1} \in \PP(a,b,c)$ equals the 
multiplicity of $h$ at $(1,1) \in \TT^2$.
\end{remark}

\begin{lemma}
\label{lem:specialfmult}
Let $\alpha,\beta,\gamma \in \KK$ 
with $\alpha + \beta + \gamma = 0$
and $k,n_1,n_2,m_1,m_2 \in \ZZ_{\ge 0}$
such that we obtain a non-constant 
homogeneous polynomial
$$ 
f 
\ := \ 
\alpha z^k
+
\beta
x^{n_1}y^{n_2}
+
\gamma x^{m_1}y^{m_2}
\ \in \ 
\KK[x,y,z].
$$
Assume that $c k/l \not\in \langle a, b \rangle$
holds whenever $l \in \ZZ_{> 1}$ is a common divisor 
of $k,n_1,n_2$ or of $k,m_1,m_2$. 
Then the multiplicity of $V(f)$ at 
$\mathbf{1} \in \PP(a,b,c)$ is at most one.
\end{lemma}

\begin{proof}
If $f$ is a monomial, then $V(f)$ is of multiplicity 
zero at $\mathbf{1}$.
If $f$ is a binomial, then it is of multiplicity one 
at $\mathbf{1}$.
So, we may assume that $\alpha, \beta, \gamma$ 
all differ from zero. 
We follow Remark~\ref{rem:toruscoord}.
Consider the homomorphism of tori
$\kappa \colon \TT^3 \to \TT^2$
and let~$u,v$ be coordinates on $\TT^2$.
Then there are monomials $u^{p_1}v^{p_2}$ 
and $u^{q_1}v^{q_2}$ with
$$
\kappa^*(u^{p_1}v^{p_2}) 
\ = \ 
\frac{x^{n_1}y^{n_2}}{z^2},
\qquad\qquad
\kappa^*(u^{q_1}v^{q_2}) 
\ = \ 
\frac{x^{m_1}y^{m_2}}{z^2}.
$$
We have $f = z^k \kappa^*(h)$ for 
$h := \alpha + \beta u^{p_1}v^{p_2} + \gamma u^{q_1}v^{q_2}$
and the multiplicity of $f$ at~$\mathbf{1}$ equals
the multiplicity of $h$ at $(1,1)$.
Assume that latter is at least two.
Then $h$ and its derivatives 
$\partial h / \partial u$ and 
$\partial h / \partial v$ vanish 
simultaneously  at $(1,1)$.
This means
$$ 
\alpha + \beta + \gamma \ = \ 0,
\qquad
\beta p_1 + \gamma q_1 \ = \ 0,
\qquad
\beta p_2 + \gamma q_2 \ = \ 0,
$$
which implies that $(p_1,p_2)$ and $(q_1,q_2)$ are 
proportional and thus 
$u^{p_1}v^{p_2}$ and $u^{q_1}v^{q_2}$
are powers of a monomial $g = u^{w_1}v^{w_2}$
with coprime exponents $w_1,w_2$.
The pullback monomials are thus of the form
$$
\frac{x^{n_1}y^{n_2}}{z^k} \ = \ \kappa^*(g)^{l_1},
\qquad\qquad
\frac{x^{m_1}y^{m_2}}{z^k} \ = \ \kappa^*(g)^{l_2}.
$$
Observe that $l_1$ divides $k, n_1, n_2$.
Thus we have 
$c k/l_1 =  an_1/l_1 + bn_2/l_1 \in \langle a,b \rangle$.
By the assumption, this means $\l_1 = 1$.
Analogously, 
$l_2$ divides $k, m_1, m_2$ 
and we conclude $l_2 = 1$.
Thus, $f$ is a binomial and vanishes of 
order one at $\mathbf{1}$.
Consequently, $h$ cannot vanish of order 
at least two at $(1,1)$.
\end{proof}

\begin{lemma}
\label{lem:2mons}
Assume that none of $a,b,c$ lies in the 
monoid generated by the other two
and that $2c$ lies in the monoid generated 
by $a$ and $b$.
Then any $0 \ne f \in S_{2c}$ 
vanishes with multiplicity at most one 
at $\mathbf{1} \in \PP(a,b,c)$.
\end{lemma}

\begin{proof}
A monomial $x^{n_1}y^{n_2}z^{n_3} \in S$ is of 
degree $2c$ if and only if it equals $z^2$
or is of the shape $x^{n_1}y^{n_2}$.
Indeed, we must have $n_3 \le 2$ and $n_3 = 1$ 
is impossible, because this means
$2c = c + an_1+bn_2$, contradicting 
$c \not\in \langle a, b \rangle$.
We obtain
$$
 g
 \ = \ 
 \alpha z^2
 +
 \beta
 x^{n_1}y^{n_2}
 +
 \gamma x^{m_1}y^{m_2}. 
 $$
with coefficients $\alpha, \beta,\gamma \in \KK$,
as~\cite[4.4, p.~80]{diophantic} tells us 
that there are at most two monomials of degree 
$2c$  only depending on $x$ and $y$.
If $\mathbf{1} \in V(f)$ holds, 
then we have 
$\alpha + \beta + \gamma = 0$
and Lemma~\ref{lem:specialfmult}
gives the assertion.
\end{proof}

\goodbreak

\begin{proposition}
\label{lem:reesdeg2nottwice}
Let $f_1 \in S_{d_1}$ and $f_2 \in S_{d_2}$ 
be an orthogonal pair.
If one of the~$f_i$ is of Rees multiplicity two,
then the other is not.
\end{proposition}

\begin{proof}
If one of $a,b,c$ lies in the monoid generated 
by the other two, then Theorem~\ref{isKstar}
and Remark~\ref{rem:orthpairdegs} give the 
assertion.
So, we only have to consider the case, 
where none of $a,b,c$ lies in the 
monoid generated by the other two.

Assume that both members $f_1,f_2$ of the orthogonal 
pair are of Rees multiplicity two.
Then, by Proposition~\ref{prop:mdscriterion2}, each 
$V(f_i) \subseteq \PP(a,b,c)$ is an irreducible 
curve and the strict transforms $C_i \subseteq X$ 
satisfy $C_1 \cdot C_2 = 0$.
Thus, Lemma~\ref{twocurves} says that $\mathbf{1}$ is 
the only intersection point of $V(f_1)$ and $V(f_2)$.

In a first step we show that each $V(f_i)$ contains 
at least one of the toric fixed points $[1,0,0]$,
$[0,1,0]$ and $[0,0,1]$.
Assume that one $V(f_i)$ does not. 
Then $f_i$ is of the shape
$$ 
f_i 
\ = \
\alpha x^{p_1} + \beta y^{p_2} + \gamma z^{p_3}
+ f_i',  
$$   
where $\alpha, \beta, \gamma \in \KK^*$ 
and the monomials of $f_i' \in S_{d_i}$ 
are all in two or three variables.
Since $f_i$ is homogeneous of degree $d_i$,
we obtain
$$ 
d_i \ = \ p_1a \ = \ p_2b \ = \ p_3c.
$$ 
As $a,b,c$ are pairwise coprime, 
$d_i = abcn$ holds with $n \in \ZZ_{\ge 1}$.
The orthogonality condition $d_1d_2 = 4abc$ 
gives $nd_j = 4$ for the $j \ne i$.
This implies $a,b,c \le 4$.
Then one of $a,b,c$ lies in the monoid
generated by the other two; a contradiction.

Thus, we saw that each of the curves $V(f_i)$ 
contains at least one toric fixed point 
and no toric fixed point is contained in 
both of them. 
After suitably reordering $a,b,c$, we are 
left with the following three cases.

\medskip
\noindent
\emph{Case 1.} 
Each of the curves $V(f_1)$ and $V(f_2)$ contains 
exactly one toric fixed point, namely
$[1,0,0]$ and $[0,1,0]$ respectively.
Then $f_1$ and $f_2$ are of the shape
$$
f_1 \ =\  \beta_1 y^{p_1} +  \gamma_1 z^{p_2} + f_1',
\qquad\qquad
f_2 \ =\  \alpha_1 x^{q_1} + \gamma_2 z^{q_2} + f_2',
$$
where $\alpha_i, \beta_i, \gamma_i \in \KK^*$ 
holds and the monomials of the $f_i' \in S_{d_i}$ 
are all in two or three variables.
The homogeneity of the $f_i$ implies
$$
d_1 = bp_1 = cp_2,
\qquad\qquad
d_2 = aq_1 = cq_2.
$$
Pairwise coprimality of $a,b,c$ gives
$d_1 = bcn$ and $d_2=acm$ with $n,m \in \ZZ_{\ge 1}$. 
The orthogonality condition $d_1d_2 = 4abc$ implies
$cnm = 4$.
We conclude $c=4$ and $n=m=1$, because 
$c \le 2$ would imply $a \in \langle b,c \rangle$ 
or $b \in \langle a,c \rangle$.
Thus, $d_1 = 4b$ holds.
Now we use Condition~\ref{ort}~(i):
$$
d_1^2 \le 4abc
\implies
16b^2 \le 16ab
\implies 
b \le a
\implies 
b < a,
$$
where the last conclusion is due to 
$a \not\in \langle b,c \rangle$.
On the other hand, 
$a \not\in \langle b,c \rangle$
implies that $a$ is less or equal 
to the Frobenius number of the monoid 
$\langle b,c \rangle$.
This means
$$
a
\ \le \ 
(c-1)(b-1)-1 
\ = \ 
(4-1)(b-1)-1 
\ = \ 
3b-4.
$$
Moreover, $a-b$ and $2b$ are even but not 
divisible by $c = 4$. Consequently, $3b-a$  
is divisible by~$4$.
We claim
$$ 
f_1 \ = \ \beta_1 y^4 +  \gamma_1 z^b + \delta xyz^{\frac{3b-a}{4}},
\qquad
\beta_1, \gamma_1, \delta \in \KK^*.
$$
Note that we need at least three terms, 
because binomials are of Rees multiplicity one. 
The task is to show that there are no further 
monomials of degree $4b$ than the ones above.
Each monomial $x^ny^mz^l$ of degree $4b$ gives an
equation
$$
an + bm + 4l \ = \  4b, 
\qquad
n,m \in \ZZ_{\ge 0}.
$$
Clearly, $m \le 4$ holds.
Because of $a > b$, we have $n \le 3$. 
As $an+bm$ is divisible by~$4$, 
the only possibilities for $(n,m)$ are
$(0,4)$, $(0,0)$ and $(1,1)$.
Having verified the special shape for $f_1$,
we can compute the multiplicity of $V(f_1)$
according to Remark~\ref{rem:toruscoord}.
The quotient map is given on the tori as 
$$ 
\kappa \colon \TT^3 
\ \to \ 
\TT^2,
\qquad\qquad
(x,y,z) 
\ \mapsto \ 
\left(
\frac{z^b}{y^4}, 
\frac{x z^{\frac{3b-a}{4}}}{y^4} 
\right).
$$
We have 
$f_1 = y^4 \kappa^*(h)$ with 
$h := \beta_1 + \gamma_1 u + \delta v$. 
The polynomial $h$ has multiplicity one at 
$(1,1)$; a contradiction.

\medskip
\noindent
\emph{Case 2.} 
The curve $V(f_1)$ contains $[1,0,0]$ and $[0,1,0]$
and $V(f_2)$ contains $[0,0,1]$.
Then $f_1$ and $f_2$ are of the shape
$$
f_1 \ = \ \gamma_1 z^{p} + f_1',
\qquad\qquad
f_2 \ = \  \alpha_2 x^{q_1} + \beta_2 y^{q_2} + f_2',
$$
where $\alpha_i, \beta_i, \gamma_i \in \KK^*$,
and the polynomials $f_i' \in S_{d_i}$ have 
only monomials in two or three variables.
By homogeneity of the $f_i$
we have  
$$
d_1= cp,
\qquad\qquad
d_2 = aq_1 = bq_2 = abn,
$$
where $n$ is a positive integer.
The orthogonality condition $d_1d_2 = 4abc$ 
provides us with $np=4$.
The case $n=p=2$ is impossible:
we would have $2c \in \bangle{a,b}$ and,
by Lemma~\ref{lem:2mons}, the multiplicity 
of~$f_1$ at $\mathbf{1} \in \PP(a,b,c)$
would be one.
We end up with $p=4$ and $n=1$.
This means $d_1=4c$ and $d_2=ab$.
Condition~\ref{ort}~(i) gives
$$
d_1^2 \le 4abc
\implies
16c^2 \le 4abc
\implies 
c \le \frac{ab}{4}.
$$
This implies $2c \not\in \langle a,b \rangle$,
because otherwise we find a binomial
$g = z^2 - x^ny^m$ of degree $2c$ and 
Rees multiplicity~$1$ which satisfies
Condition~\ref{ort}~(i), contradicting
the minimality of the degree of~$f_1$.
We determine $f_1$ more explicitly.
Each monomial $x^ny^mz^l$ of degree $4c$ gives an
equation
$$
an + bm + lc \ = \  4c, 
\qquad
n,m,l \in \ZZ_{\ge 0}.
$$
Here, $l = 2,3$ are excluded because of 
$2c \not\in \langle a,b \rangle$ and
$c \not\in \langle a,b \rangle$.
Thus, we have $l \le 1$.
If $4c < ab$ holds,
then we can apply~\cite[4.4, p.~80]{diophantic}
and obtain that there is at most one monomial 
of the form $zx^{n_1}y^{n_2}$ 
and at most one of the form $x^{m_1}y^{m_2}$ 
in degree~$4c$.
Thus, we have 
$$
f_1 
\ = \  
\alpha z^4 + \beta zx^{n_1}y^{n_2}  + \gamma x^{m_1}y^{m_2}
$$
and Lemma~\ref{lem:specialfmult} tells us 
that $V(f_1)$ is multiplicity one at $\mathbf{1}$;
a contradiction.
We are left with discussing the case $4c = ab$. 
By coprimality of $a$ and $b$, we obtain
$a = 4a'$ or $b = 4b'$. Thus, $c = a'b$ and 
$c = b'a$, both contradicting 
$c \not\in \langle a,b \rangle$. 
Thus, Case~2 cannot occur.

\medskip
\noindent
\emph{Case 3.} 
The curve $V(f_1)$ contains $[1,0,0]$  
and $V(f_2)$ contains $[0,1,0]$ and $[0,0,1]$.%
Then $f_1$ and $f_2$ are of the shape
$$
f_1 \ = \ \beta_1 y^{p_1} + \gamma_1 z^{p_2} + f_1',
\qquad\qquad
f_2 \ = \  \alpha_2 x^{q}  + f_2',
$$
where $\alpha_i, \beta_i, \gamma_i \in \KK^*$,
and the polynomials $f_i' \in S_{d_i}$ have 
only monomials in two or three variables.
By homogeneity of the $f_i$
we have  
$$
d_1 \ = \ bp_1 \ = \ cp_2 \ = \ bcn,
\qquad\qquad
d_2 \ = \ aq,
$$
where $n$ is a positive integer. The orthogonality 
condition $d_1d_2 = 4abc$ gives $nq = 4$.
We obtain $q=4$ and $n=1$, 
because $q = 1$ is excluded 
by $a \not\in \langle b,c \rangle$ and $q=2$ 
is impossible due to Lemma~\ref{lem:2mons}.
Thus, we have $d_1 = bc$ and $d_2 = 4a$.
Condition~\ref{ort}~(i) gives
$$
d_1^2 \le 4abc
\implies
b^2c^2 \le 4abc
\implies 
bc \le 4a.
$$
We have $2a \not\in \langle b,c \rangle$,
because otherwise, there is a binomial $f_2' = x^2 - y^{n_1}z^{n_2}$ 
of degree $d_2' = 2a$ and Rees multiplicity $\mu_2' = 1$
satisfying the orthogonality condition; a contradiction
to the minimality of the degree of $f_2$.   
In particular, $2a$ is less than the Frobenius number
of $\langle b, c\rangle$ which means
$$ 
2a 
\ \le \ 
(b-1)(c-1) -1 
\ = \ 
bc - b -c.
$$
Combining this with the previous estimate, we obtain
$b+c \le 2a$.
We determine $f_2$ explicitly. 
Any monomial $x^ly^mz^n$ of degree 
$d_2 = 4a$ gives rise to an equation
$$ 
la + mb + nc \ =  \ 4a.
$$
We search for solutions with $l \le 3$. 
The cases $l=3,2$ are excluded because of 
$a \not\in \langle b, c \rangle$ and 
$2a \not\in \langle b, c \rangle$.
Thus, we look for pairs $m,n \in \ZZ_{\ge 0}$ 
satisfying one of the equations
$$ 
mb + nc \ =  \ 4a,
\qquad\qquad
mb + nc \ =  \ 3a.
$$
Consider the case $b \not\in \{2,3,4\}$.
Then $b$ does not divide $ka$ for
$k=1,2,3,4$.
Fix positive integers $u,v$ with 
$ub-vc = 1$.
Then~\cite[Corollary 1.6]{contdisc} says
that the number $\xi_{b,c}(ka)$ of pairs 
$(m,n) \in \ZZ^2_{\ge 0}$ satisfying 
$mb+nc = ka$ is given as 
$$
\xi_{b,c}(ka) 
\  = \  
\Big\lfloor\frac{uka}{c}\Big\rfloor 
- 
\Big\lfloor\frac{vka}{b}\Big\rfloor,
\qquad
\text{for }
k = 1,2,3,4.
$$
As just seen, we have $\xi_{b,c}(a) = 0$
and $\xi_{b,c}(2a) = 0$. 
The first equality implies that the two 
numbers $ua/c$ and $va/b$ lie in some open 
interval $]s,\, s+1[$, where $s \in \ZZ$.
The second equality implies that both
numbers even lie either in  
$]s,\, s+1/2[$ or in $]s+1/2,\, s+1[$. 
We obtain 
$$ 
\xi_{b,c}(3a) \ \le \ 1
\qquad\qquad
\xi_{b,c}(4a) \ \le \ 1.
$$
In other words, there are at most three 
monomials in $S_{4a}$, namely $x^4$, $xy^{n_1}z^{n_2}$ 
and $y^{m_1}z^{m_2}$. 
Lemma~\ref{lem:specialfmult} says that $V(f_2)$ is 
of multiplicity at most one at $\mathbf{1}$; 
a contradiction. 
Analogously, the case $c \not\in \{2,3,4\}$
is excluded.
Thus, we are left with $b,c \in \{2,3,4\}$.
But this impossible due to 
$a \not\in \langle b,c \rangle$ and 
$2a \not\in \langle b,c \rangle$.
\end{proof}

\section{Proof of Theorem~\ref{1112}}
\label{sec:proofs3}

We will use the following general criterion 
for verifying Cox ring generators.
Consider an arbitrary 
Mori dream space $X_1$ and the blow-up $X_2$ 
of an irreducible subvariety $C \subseteq X_1$ 
contained in the smooth locus of~$X_1$.
We will denote by $I \subseteq R_1:=\Cox(X_1)$ 
the homogeneous ideal corresponding to $C \subseteq X_1$
and by $J \subseteq R_1$ the irrelevant ideal.
The morphism $X_2 \to X_1$ defines a canonical pull back 
map $R_1 \to R_2 := \Cox(X_2)$ of Cox rings.
We ask if for a given choice of homogeneous 
generators $f_1, \ldots, f_k \in I$ for $I$,
the canonical section $t \in R_2$ of the exceptional 
divisor $E \subseteq X_2$ together with 
$f_it^{-m_i}$, where $i = 1, \ldots, k$ and $m_i$
is the Rees multiplicity, 
generate the Cox ring $R_2$ of $X_2$ 
as an $R_1$-algebra.

\begin{proposition}
\label{prop:ismds}
In the above situation, 
let $g_1,\ldots,g_m$ be homogeneous
generators of the $\KK$-algebra $R_1$
and let $f$ be the product over all 
$g_j$ not belonging to $I$.
Set 
$$
B_0 
\ := \ 
\{t^{m_i}s_i - f_i; \ i = 1, \ldots,  k\}
\ \subseteq \ 
R_1[s_1,\ldots,s_k,t].
$$ 
Then $R_2$ is generated as a $\KK$-algebra 
by $t$,
the $f_it^{-m_i}$, where $i = 1, \ldots, k$,
and the~$g_j$ not belonging to $I$,
provided that there is a finite set 
$B_0 \subseteq B \subseteq \langle B_0 \rangle : \langle t\rangle^\infty$
with
$$
\dim(R_1) 
\ = \
\dim (\langle B\cup\{t\}\rangle)
\ > \
 \dim  (\langle B\cup\{t,f\}\rangle).
$$
Moreover, in this case, the Cox ring $R_2$ of 
$X_2$ is isomorphic as a $\Cl(X_2)$-graded algebra 
to 
$$
R_1[s_1,\dots,s_k,t] 
/ 
\bigl( \langle B \rangle : \langle t \rangle^{\infty} \bigr).
$$
\end{proposition}

\begin{proof}
Recall that the Cox ring $R_2$ of $X_2$ is
the saturated Rees algebra $R_1[I]^{\rm sat}$.
As before, let $m_i$ be the Rees multiplicity of 
$f_i$ for $i= 1, \ldots, k$. 
The kernel of the $\Cl(X_2)$-graded homomorphism
$$
\psi 
\colon 
R_1[s_1,\dots,s_k,t] \ \to \ R_1[I]^{\rm sat},
\qquad
 s_i \ \mapsto \ f_it^{-m_i},
\qquad
 t \ \mapsto\ t
$$
is the saturation $I_2 :=  I_2' : \langle t \rangle^\infty$,
where we set $I_2' := \langle  B \rangle$.
Observe that the dimension of $R_1$ 
equals that of $I_2 +\langle t\rangle$.
Thus, by our assumption, we have 
$$ 
\dim(I_2 +\langle t\rangle)
\, = \, 
\dim(R_1) 
\, = \,
\dim(I_2' + \langle t \rangle)
\, > \, 
\dim(I_2' + \langle t, f \rangle) 
\, \ge \, 
\dim(I_2 + \langle t, f \rangle).
$$
Consequently, we meet the condition 
of~\cite[Algorithm 5.4]{hkl}
which guarantees that the homomorphism~$\psi$ 
is surjective.
The assertion follows.
\end{proof}

\begin{proof}[Proof of Theorem~\ref{1112}]
We show that~(i) implies~(ii).
First note that the Cox ring of
$X$ finitely generated.
Indeed, $\mathcal{R}(X)$ is 
the saturated Rees algebra $S[I]^{\rm sat}$
which, under the assumption~(i), is generated
by $t^{-1}$, the Cox ring 
generators $x,y,z$ of $\PP(a,b,c)$, the 
elements $g_i t^{-1}$,
where the $g_i$ generate the ideal $I:J^\infty$,
and the $h_j t^{-2}$, where the $h_j$ 
generate the ideal $I^2:J^\infty$.
Thus, Proposition~\ref{prop:mdscriterion2} 
provides us with an orthogonal pair 
$f_1,f_2$, where $f_i \in S_{d_i}$ is 
of Rees multiplicity $\mu_i$.
Remark~\ref{rem:orthpairdegs} says that
$(d_1,-\mu_1)$ and $(d_2,-\mu_2)$ 
occur in the set of $\Cl(X)$-degrees 
of any system of generators of the 
Cox ring $\mathcal{R}(X)$.
Thus, by assumption, we have $\mu_i \le 2$.
Proposition~\ref{lem:reesdeg2nottwice}
yields that $\mu_i = 2$ holds at most 
once.

For both $f_i$, their degree $d_i$ is positive
and thus also their Rees multiplicity $\mu_i$
is positive.
Since we assume none of $a,b,c$ to lie in the 
monoid generated by the other two,
the case $\mu_1 = \mu_2$ is excluded by 
Lemmas~\ref{lem:orthpairdeg1} and~\ref{reesdeg2}.
We now consider the case $\mu_1 = 1$ and $\mu_2 = 2$.
Then we may assume 
$$ 
f_1 
\ = \ 
x^{p_1}-y^{p_2}z^{p_3},
\qquad
\qquad
 f_2
\ = \ 
\alpha y^{q_1} + \beta z^{q_2} + f_2',
$$
where $\alpha, \beta \in \KK^*$ and $f_2' \in S_{d_2}$ 
has only monomials in two or three variables.
Indeed, Lemma~\ref{lem:orthpairdeg1} say that we may 
assume $f_1$ to be a binomial.
By Proposition~\ref{prop:mdscriterion2}, the binomial
$f_1$ is prime, and thus we may assume it to be 
of the displayed shape.
In particular, the points $[0,1,0]$ and $[0,0,1]$ 
are contained in $V(f_1)$.
Lemma~\ref{twocurves} tells us that none of these 
two points lies in $V(f_2)$ and thus, $f_2$ must be 
of the above shape.
Homogeneity of $f_1$, $f_2$ and the orthogonality 
condition~\ref{ort}~(ii) lead to the equations
$$
ap_1 \ = \ bp_2 + cp_3,
\qquad
bq_1 \ = \ cq_2,
\qquad
p_1q_1 \ = \ 2c.
$$
Since $b$ and $c$ are coprime, 
the second equation shows that $q_1 = lc$
holds with $l \in \ZZ_{\ge 1}$. 
Substituting this in the last equation
gives $lp_1 = 2$.
Because of $a \not\in \langle b,c \rangle$,
we have $p_1 \ne 1$ and thus obtain
$p_1=2$ and $l=1$. 
Consequently, $q_1=c$ and $q_2=b$ hold. 
With $n := p_2$ and $m := p_3$, the first equation thus becomes
$$
 2a \ = \ nb + mc.
$$
We now describe the polynomial $f_2'$ in more detail.
First, we determine the monomials
$x^ky^pz^q$ of degree $d_2 = bc$.
This means to look at the equation
$ka+bp+cq = bc$, which implies 
$$
b(kn+2p) + c(km+2q)
\ = \ 
2bc.
$$
In particular, $kn+2p = rc$ holds for some 
integer $r \geq 1$. 
Substituting this in the displayed
equation, we obtain $km+2q = (2-r)b$.
This implies $r \leq 1$ and thus $r=1$.
Thus, we arrive at
$$
p \ = \ \frac{c-kn}{2},
\qquad\qquad
q \ = \ \frac{b-km}{2}.
$$
In particular, we see that $k$ must be odd,
as $b$ and $c$ are coprime. 
Up to now, we are able to express the 
possible monomials of degree $d_2 = bc$ 
in terms of $k,m,n$ and $b,c$.
We are going to apply Remark~\ref{rem:toruscoord}.
As a homomorphism of tori we take
$$ 
\kappa \colon \TT^3 \ \to \ \TT^2,
\qquad\qquad
(x,y,z) 
\ \mapsto \ 
\left(
\frac{y^c}{z^b}, 
\frac{xy^{\frac{c-n}{2}}z^{\frac{b-m}{2}}}{z^b} 
\right).
$$
Then, with the coordinates $u,v$ on $\TT^2$
and $l := (k-1)/2$, we 
can write the general monomial of degree 
$d_2 = bc$ as
$$ 
x^ky^{\frac{c-kn}{2}}z^{\frac{b-km}{2}}
\ = \ 
z^b \kappa^* \!\! \left(\frac{v^k}{u^{\frac{k-1}{2}}}\right)
\ = \ 
z^b \kappa^* \!\! \left(\frac{v^{2l+1}}{u^l}\right).
$$
Consequently, with suitable coefficients $\gamma_l \in \KK$, 
we can write $f_2 =  z^b h$ with a Laurent polynomial 
$$ 
h 
\ = \ 
\alpha 
+ 
\beta u
+
\sum_{l=0}^s \gamma_l \frac{v^{2l+1}}{u^l}.
$$
The fact that $f_2$ is of Rees multiplicity two implies 
that $h$ as well as its first order partial 
derivatives $\partial h / \partial u$
and $\partial h / \partial v$ vanish at $(1,1)$.
This leads to the conditions
$$ 
\alpha + \beta + \gamma_0 + \ldots + \gamma_s \ = \ 0,
$$
$$
\beta - \gamma_0 - \ldots - (s+1) \gamma_s   \ = \ 0,
$$
$$
\gamma_0 + 3 \gamma_1 + \ldots + (2s+1) \gamma_s \ = \ 0.
$$
In particular, we see that the polynomial $f_2$ must have
at least four terms. In fact, we can choose it to be
$$ 
f_2
\ = \  
y^c 
+ z^b 
- 3xy^{\frac{c-n}{2}}z^\frac{b-m}{2}
+ x^3y^{\frac{c-3n}{2}}z^\frac{b-3m}{2}.
$$
As all exponents are non-negative, we see that in 
the equation $2a = bm + mc$ we have 
$b \ge 3m$ and $c \ge 3n$.
Thus, we verified the conditions of~(ii)
in the case $\mu_1 = 1$ and $\mu_2=2$.
If $\mu_2 = 1$ and $\mu_1=2$ holds, then
we may proceed exactly the same way; 
observe that we only made use of the 
orthogonality condition~\ref{ort}~(ii).

We show that~(ii) implies the supplement.
First we claim that the ideal $I \subseteq S$ 
is generated by the binomials
$$
f_1 \ = \ x^2 - y^nz^m,
\qquad
f_2 \ := \ xz^\frac{b-m}{2}-y^\frac{c+n}{2},
$$
$$
f_3 
\ := \ 
xy^\frac{c-n}{2}-z^\frac{b+m}{2}
\ = \ x^{-1}(y^\frac{c-n}{2}f_1-z^nf_2).
$$
Indeed, from~\cite[Lemma~7.6]{MiSt} we infer that
$I$ equals the saturation
$\langle f_1,f_2\rangle : \langle xyz\rangle^\infty$.
Now, $f_3$ lies in the saturation and 
$\langle f_1,f_2,f_3\rangle$ is prime,
which gives the claim.
Observe that we have 
$$
f_4 
\ := \ 
xy^\frac{c-3n}{2}z^\frac{b-3m}{2}f_1-y^\frac{c-n}{2}f_2-z^\frac{b-m}{2}f_3
\ \in \ 
(I^2 : J^\infty)\setminus I^2.
$$
We want to show that the Cox ring of $X$
is generated by the canonical 
section $t$ of the exceptional divisor,
the pull back sections $x,y,z$,
the sections $s_i := f_it^{-1}$ for $i = 1,2,3$ 
and $s_4 := f_4t^{-2}$.
This is equivalent to saying that
the Cox ring of $X$ is isomorphic to
$\KK[x,y,z,s_1,\dots,s_4,t]/I_2$,
where 
$$
 I_2 
\ := \ 
\langle s_1t-f_1,\ s_2t-f_2,s_3t-f_3,\ s_4t^2-f_4\rangle : t^\infty.
$$
The localization 
$(I_2)_t \subseteq \KK[x,y,z,s_1,\dots,s_4,t]_t$
is a prime ideal of dimension four 
and thus $I_2$ is a prime ideal of dimension 
four.   
Moreover, the ideal $I_2$ contains the
ideal $I_2'$ generated by the 
following polynomials
\begin{gather*}
  f_1 - s_1t,\qquad
  f_2 - s_2t,\qquad
  f_3 - s_3t,\\
  y^\frac{c-n}{2}s_1-z^ms_2-xs_3,\qquad
  z^\frac{b-m}{2}s_1-xs_2-y^ns_3,\\
  xy^\frac{c-3n}{2}z^\frac{b-3m}{2}s_1-y^\frac{c-n}{2}s_2-z^\frac{b-m}{2}s_3-s_4t,\\
  s_3^2 + y^\frac{c-3n}{2}s_1s_2-z^ms_4,\\
  s_2^2 + z^\frac{b-3m}{2}s_1s_3-ys_4,\\
  y^\frac{c-3n}{2}z^\frac{b-3m}{2}s_1^2 - s_2s_3-xs_4.
\end{gather*}
Let $I_2'' \subseteq \KK[x,y,z,s_1,\ldots,s_4]$
be the ideal generated by the polynomials 
obtained from the above ones by setting $t := 0$.
Then the first three generators of $I_2''$
are $f_1,f_2,f_3$. 
We take a look at the zero set $V(I_2'') \subseteq \KK^7$.
First consider the area $W_0 \subseteq V(I_2'')$ cut out by
$xyz = 0$. 
By the nature of $f_1,f_3,f_3$, each of $x,y,z$ 
vanishes identically on $V(I_2'')$ and we see that 
$W_0 = V(x,y,z,s_2,s_3)$ is of dimension two.
Now consider the set of points $W_1 \subseteq V(I_2'')$
satisfying $xyz \ne 0$.
We have a finite surjection
$$ 
\KK^* \times \KK^4
\ \to \ 
V(f_1,f_2,f_3),
\qquad 
(\xi,s_1,s_2,s_3,s_4) 
\ \mapsto \ 
(\xi^a,\xi^b,\xi^c,s_1,s_2,s_3,s_4).
$$
The image contains $W_1$ and the pullback 
of the generators number 4,5 and 6 of 
$I_2''$ are multiples of 
$$
\xi^\frac{bc-bn-2a}{2}s_1-\xi^{cm-a}s_2 - s_3
\ \in \ 
\KK[\xi^{\pm 1}, s_1,s_2,s_3,s_4].
$$
Now, we eliminate $s_3$ by means of this 
relation and see that turns the pullbacks 
of the remaining three generators of $I_2''$ 
are multiples of a common polynomial,
depending on~$s_4$.
We conclude that $W_1$ is of dimension three.
Altogether, we verified that 
$I_2' + \langle t \rangle$ has a three-dimensional 
zero set
and $I_2' + \langle t, xyz \rangle$ a two-dimensional
one.
Thus, we can apply Proposition~\ref{prop:ismds}
to see that the Cox ring of $X$ is as claimed.

To conclude the whole proof, it suffices to show that 
the supplement implies~(i). But this is 
obvious.
\end{proof}

\begin{remark}
\label{mist}
Observe that in the proof of Theorem~\ref{1112},
the fourth and fifth generators of the ideal
$I_2'$ come from the following syzygies 
of the lattice ideal $\<f_1,f_2,f_3\>$ as found by 
the methods from~\cite[Chap.~9]{MiSt}:

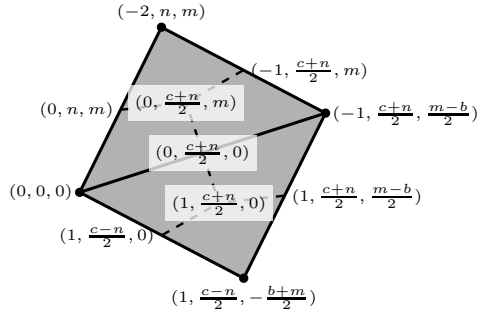
\begin{center}
\tiny
\begin{tikzpicture}[scale=1,rotate=18]
 \coordinate (x1) at (-1.7,0);
 \coordinate (x2) at (1.7,0);
 \coordinate (x3) at (0,1.75);
 \coordinate (x4) at (0,-1.75);

 \fill[color = black!30] (x1) -- (x3) -- (x2) -- (x4) -- cycle;
 
 \draw[line width=1.15pt] (x1) -- (x2) -- (x3) -- (x1);
 \draw[line width=1.15pt] (x1) -- (x4) -- (x2);
 
 \fill (x1) circle (1.75pt);
 \fill (x2) circle (1.75pt);
 \fill (x3) circle (1.75pt);
 \fill (x4) circle (1.75pt);
 
 % vertices
 \draw (x1) node[anchor=east]{$(0,0,0)$};
 \draw (x2) node[anchor=west]{$(-1,\tfrac{c+n}{2},\tfrac{m-b}{2})$};
 \draw (x3) node[anchor=south]{$(-2,n,m)$};
 \draw (x4) node[anchor=north]{$(1,\tfrac{c-n}{2},-\tfrac{b+m}{2})$};
 
  % edges
 \coordinate (x12) at ($(x1)!.5!(x2)$);
 \coordinate (x13) at ($(x1)!.5!(x3)$);
 \coordinate (x14) at ($(x1)!.5!(x4)$);
 \coordinate (x23) at ($(x3)!.5!(x2)$);
 \coordinate (x24) at ($(x4)!.5!(x2)$);
  
 % 2-faces:
 \coordinate (x123) at (0,0.7);
 \coordinate (x124) at (0,-0.7);
 
  % arrows
  \draw [black,shorten <=0.01cm,shorten >=0.01cm, dashed, color = black, line width = .9pt] (x12) -- (x123);
 \draw [black,shorten <=0.01cm,shorten >=0.01cm, dashed, color = black, line width = .9pt] (x13) -- (x123);
 \draw [black,shorten <=0.01cm,shorten >=0.01cm, dashed, color = black, line width = .9pt] (x23) -- (x123);
 
  \draw [black,shorten <=0.01cm,shorten >=0.01cm, dashed, color = black, line width = .9pt] (x12) -- (x124);
 \draw [black,shorten <=0.01cm,shorten >=0.01cm, dashed, color = black, line width = .9pt] (x14) -- (x124);
 \draw [black,shorten <=0.01cm,shorten >=0.01cm, dashed, color = black, line width = .9pt] (x24) -- (x124);
%  \draw [black,shorten <=0.3cm,shorten >=0.3cm, ->, line width = 1.5pt, color = black!60] (x12) -- (x123);
%  \draw [black,shorten <=0.1cm,shorten >=0.7cm, ->, line width = 1.5pt, color = black!60] (x13) -- (x123);
%  \draw [black,shorten <=0.1cm,shorten >=0.7cm, ->, line width = 1.5pt, color = black!60] (x23) -- (x123);
%  
%   \draw [black,shorten <=0.3cm,shorten >=0.3cm, ->, line width = 1.5pt, color = black!60] (x12) -- (x124);
%  \draw [black,shorten <=0.1cm,shorten >=0.85cm, ->, line width = 1.5pt, color = black!60] (x14) -- (x124);
%  \draw [black,shorten <=0.1cm,shorten >=0.85cm, ->, line width = 1.5pt, color = black!60] (x24) -- (x124);

%  \fill[color=black!30] (x12) ellipse (14pt and 5pt);
 \draw (x12) node[fill=white,opacity=.75]{\phantom{$(0,\tfrac{c+n}{2},0)$}};
 \draw (x12) node{$(0,\tfrac{c+n}{2},0)$};
 \draw (x13) node[anchor=east]{$(0,n,m)$};
 \draw (x14) node[anchor=east]{$(1,\frac{c-n}{2},0)$};
 \draw (x23) node[anchor=west]{$(-1,\tfrac{c+n}{2},m)$};
 \draw (x24) node[anchor=west]{$(1,\tfrac{c+n}{2},\tfrac{m-b}{2})$};

 \draw (x123) node[fill=white,opacity=.75]{\phantom{$(0,\tfrac{c+n}{2},m)$}};
 \draw (x124) node[fill=white,opacity=.75]{\phantom{$(1,\tfrac{c+n}{2},0)$}};
 \draw (x123) node{$(0,\tfrac{c+n}{2},m)$};
 \draw (x124) node{$(1,\tfrac{c+n}{2},0)$};
\end{tikzpicture}
\end{center}
\end{remark}

\begin{corollary}
\label{cor:3bc}
Consider a triple $(3,b,c)$ such that none 
of the entries lies in the monoid generated 
by the other two.
Then the Cox ring of $X(3,b,c)$ is as
in Theorem~\ref{1112}.
\end{corollary}

\begin{proof}
It suffices to show that  $(3,b,c)$
satisfies condition~\ref{1112}~(ii).
To see this, observe that if $b < c$ then
also $c < 2b$ holds; otherwise,
$c$ would be in the semigroup
$\langle 3,b\rangle$ as it is bigger that
the Frobenius number $2(b-1)+1$
of the semigroup. Also observe that 
the equation $b+c\equiv 0\pmod 3$ 
must hold, since otherwise $c$ would
belong to the semigroup
$\langle 3,b\rangle$. We deduce that
there exists a positive integer $n$ such 
that $2b = 3n + c$. Moreover, from
$c+3n = 2b < 2c$, we deduce $c>3n$.
\end{proof}

%%%%%%%%%%%%%%%%%%%%%%%%%%%%%%%%%%%%

\section{Algorithms and applications}
\label{sec:algos}

Our first algorithm applies to blow ups of 
arbitrary Mori dream spaces.
We work in the setting of 
Proposition~\ref{prop:ismds}.
Based on the criterion given there,
we are able to avoid the (involved)
computation of saturations performed
in the related~\cite[Algorithm~5.6]{hkl}.

\begin{algorithm}[Verify generators]
\label{ismds}
{\em Input:} 
homogeneous generators $g_1, \ldots, g_m$ for 
the Cox ring $R_1$ of a Mori dream space~$X_1$
and homogeneous generators $f_1,\dots,f_k \in R_1$ 
of the ideal $I$ of an irreducible subvariety 
$C \subseteq X_1$ contained in the smooth locus.
\begin{itemize}
\item
For each $f_i$, compute the maximal $m_i \in \ZZ_{\ge 0}$ 
with $f_i \in I^{m_i} : J^\infty$.
\item
Let $f$ be the product of all the generators $g_i$ 
which do not vanish along $C$.
\item
Set 
$\mathfrak B 
:= 
\{t^{m_i}s_i-f_i;  i = 1, \ldots, k\}
\subseteq 
R_1[s_1,\ldots,s_k,t]$.
\item
Repeat
\begin{itemize}
\item
if 
$\mathfrak B' := \langle\mathfrak B\rangle : t^\infty 
\setminus \langle \mathfrak B \rangle$ 
is nonempty, enlarge $\mathfrak B$ by an element
of $\mathfrak B'$. 
\item
if 
$
\dim(R_1) 
= 
\dim(\langle\mathfrak B\cup \{t\}\rangle)
$
and 
$\dim(\langle\mathfrak B\cup \{t\}\rangle) 
>
\dim(\langle\mathfrak B\cup \{t,f\}\rangle)$
then return {\em true}.
\end{itemize}
until 
$
\langle\mathfrak B\rangle:t^\infty
= 
\langle\mathfrak B\rangle
$.
\item
Return {\em false}.
\end{itemize}
{\em Output: } true is returned if and only if
the Cox ring $R_2$ of the blow-up $X_2$ of $X_1$
along $C$ is generated by $t$ and  
$f_1t^{-m_1},\ldots,f_kt^{-m_k}$ as an $R_1$-algebra.
\end{algorithm}

\begin{proof}
If the algorithm returns ``true'' that 
Proposition~\ref{prop:ismds} guarantees
that $R_2$ is generated by $t$ and
$f_1^{-m_1},\ldots,f_k^{-m_k}$ as an 
$R_1$-algebra.
Conversely, assume thast $R_2$ is generated 
by $t$ and $f_1t^{-m_1},\ldots,f_kt^{-m_k}$ as an 
$R_1$-algebra.
Then the list of all $g_j,f_it^{-m_i},t$ 
comprises a system of pairwise 
$\Cl(X)$-coprime generators for $R_2$ and thus, 
the dimension conditions are fulfilled 
if $\langle \mathfrak B \rangle$ equals 
the defining ideal of $R_2$ which in turn 
is given as 
$\langle\mathfrak B\rangle:t^\infty$.
Consequently, the algorithm returns true.
\end{proof}

\begin{remark}
In the fifth line of Algorithm~\ref{ismds},
as in Remark~\ref{mist},
elements of $\mathfrak B'$ can be obtained by 
determining syzygies among (products of) the~$f_i$.
\end{remark}

The next algorithm implements 
Proposition~\ref{prop:mdscriterion2} and 
provides a Mori dreamness test in our 
concrete setting, i.e., the blow-up 
$X=X(a,b,c)$ of the point 
$[1,1,1] \in \PP(a,b,c)$.
As before, $I \subseteq S$ is the 
ideal of $[1,1,1]$ in the Cox ring $S=\KK[x,y,z]$ of
$\PP(a,b,c)$.

\begin{algorithm}[Mori dreamness test]
\label{algo:ismds}
{\em Input:} 
pairwise coprime positive integers $(a,b,c)$.
\begin{itemize}
\item
Compute a system $\mathfrak B$ of 
homogeneous generators of the ideal 
$I \subseteq S=\KK[x,y,z]$ of 
$[1,1,1] \in \PP(a,b,c)$.
\item
For $m=2,3,\ldots$ do
\begin{itemize}
\item
Compute the normal form of a basis of
$A_m := I^m : J^\infty$ with respect
to $A_{<m} := A_1A_{m-1}+\dots
+A_{\lfloor\frac{m}{2}\rfloor}A_{\lceil\frac{m}{2}\rceil}$,
select the elements of minimal degree and add
them to $\mathfrak B$.
\item
If $\mathfrak B$ contains an orthogonal 
pair $f_1,f_2 \in S$ as in Definition~\ref{ort},
then return {\em true}.
\end{itemize}
\end{itemize}
\emph{Output: } \emph{true}; 
this is returned if and only if the algorithm terminates 
and in this case, $X(a,b,c)$ is a Mori dream surface.
\end{algorithm}

\begin{proof}
If the algorithm terminates,
then it returns ``true'' and thus 
there is an orthogonal pair in $S$.
Proposition~\ref{prop:mdscriterion2} 
then yields that $X(a,b,c)$ is a Mori dream 
surface.
If $X(a,b,c)$ is a Mori dream surface,
then $\mathfrak B$ will give rise to 
a system of homogeneous generators for the 
Cox ring at some point and 
Remark~\ref{rem:orthpairdegs}
ensures that there is an orthogonal 
pair in $\mathfrak{B}$.
\end{proof}

Finally, we discuss the applications to the 
investigation of the Mori dream space property 
for $\overline{M}_{0,n}$.
Recall the following from~\cite{CaTe} 
and~\cite[Theorem~4.1]{gk}.

\begin{method}[Castravet/Tevelev]
\label{method:CastravetTevelev}
Given $n\in \ZZ_{\geq 6}$, let 
$N'\subseteq N$ be a saturated sublattice of $N:=\ZZ^{n-3}$
of rank $n-5$ generated by subsets of
$M:=\{\pm p;\ p\in \{0,1\}^{n-3}\}\subseteq N$ such that
for the quotient map
$\pi\colon N\to N/N'$,
there are $v_1,v_2,v_3\in M$ with
$$
\left\<\pi(v_1),\pi(v_2),\pi(v_3)\right\>\ =\ N/N'.
$$
Further assume that there are pairwise coprime positive integers $a,b,c$
with $av_1+bv_2+cv_3 \equiv 0 \pmod{N'}$.
If the blow up $X(a,b,c)$ of $\PP(a,b,c)$ in the point 
$[1,1,1]$ is not a Mori dream space, then also 
$\overline M_{0,n}$ is not a Mori dream space.
\end{method}

\begin{proof}[Proof of Addendum~\ref{thm:add}]
There are proper surjective morphisms
$\overline M_{0,n} \to \overline M_{0,n-1}$.
Consequently, if $\overline M_{0,n}$ is 
a Mori dream space, then also 
$\overline M_{0,n-1}$ is one.  
Thus, it suffices to show that 
$\overline M_{0,10}$ is not a Mori dream space.

The defining fan of the Losev-Manin space $\overline{L}_{10}$ 
lives in $N := \ZZ^7$ and its rays are the cones 
generated by the vectors having either 
all their coordinates in $\{0,1\}$ or in 
$\{0,-1\}$.
Consider the linear map $\ZZ^7 \to \ZZ^2$ 
given by the matrix
$$
P
\ := \ 
\left[
\begin{array}{rrrrrrr}
1 & 0 & 1 & -2 & -1 & 1 & 0
\\
0 & 1 & -1 & -3 & -2 & 2 & 1
 \end{array}
\right]
$$
A $\ZZ$-basis for the kernel $N' \subseteq N$ of $P$ 
is given by the following five primitive generators 
of the fan of $\overline{L}_{10}$:  
$$ 
e_1+e_2+e_4+e_6,
\quad
e_1+e_2+e_5+e_7,
\quad
-(e_1+e_4+e_6+e_7),
$$
$$
e_5+e_6,
\qquad
-(e_2+e_3+e_4+e_6+e_7).
$$
Moreover, the primitive generators 
$-(e_4+e_5)$, $-(e_1+e_3+e_6)$
and $e_1+e_3+e_4+e_5$ are mapped 
to the columns of 
$$
\left[
\begin{array}{rrr}
3 &-3 &-1
\\
5 &-1 &-6
 \end{array}
\right]
$$
which in turn generate the fan of $\PP(17,13,12)$.
In particular, we have a rational toric morphism from 
$\overline{L}_{10}$  to $\PP(17,13,12)$.
By~\cite[Theorem 1.5]{gk}, the surface
$X(17,13,12)$ is not Mori dream. 
Thus, Method~\ref{method:CastravetTevelev} gives the 
assertion.
\end{proof}

\begin{remark}
Method~\ref{method:CastravetTevelev} fails for 
$\overline M_{0,n}$, where $n= 7,8,9$.
In these cases, for all possible projections 
$\pi$ and the possible associated $X(a,b,c)$,
Algorithm~\ref{algo:ismds} is feasible and shows 
that the $X(a,b,c)$ are Mori dream surfaces.
So, it remains open whether $\overline M_{0,n}$
is a Mori dream space for $n= 7,8,9$.
\end{remark}


\begin{bibdiv}
\begin{biblist}

\bib{ArDeHaLa}{book}{
   author={Arzhantsev, Ivan},
   author={Derenthal, Ulrich},
   author={Hausen, J{\"u}rgen},
   author={Laface, Antonio},
   title={Cox rings},
   series={Cambridge Studies in Advanced Mathematics},
   volume={144},
   publisher={Cambridge University Press, Cambridge},
   date={2015},
   pages={viii+530},
   isbn={978-1-107-02462-5},
   review={\MR{3307753}},
}



\bib{contdisc}{book}{
   author={Beck, Matthias},
   author={Robins, Sinai},
   title={Computing the continuous discretely},
   series={Undergraduate Texts in Mathematics},
   edition={2},
   note={Integer-point enumeration in polyhedra;
   With illustrations by David Austin},
   publisher={Springer, New York},
   date={2015},
   pages={xx+285},
   isbn={978-1-4939-2968-9},
   isbn={978-1-4939-2969-6},
   review={\MR{3410115}},
   doi={10.1007/978-1-4939-2969-6},
}


\bib{Ca}{article}{
   author={Castravet, Ana-Maria},
   title={The Cox ring of $\overline M_{0,6}$},
   journal={Trans. Amer. Math. Soc.},
   volume={361},
   date={2009},
   number={7},
   pages={3851--3878},
   issn={0002-9947},
   review={\MR{2491903}},
   doi={10.1090/S0002-9947-09-04641-8},
}


\bib{CaTe}{article}{
   author={Castravet, Ana-Maria},
   author={Tevelev, Jenia},
   title={$\overline{M}_{0,n}$ is not a Mori dream space},
   journal={Duke Math. J.},
   volume={164},
   date={2015},
   number={8},
   pages={1641--1667},
   issn={0012-7094},
   review={\MR{3352043}},
   doi={10.1215/00127094-3119846},
}

\bib{chmr}{article}{
   author={Ciliberto, Ciro},
   author={Harbourne, Brian},
   author={Miranda, Rick},
   author={Ro{\'e}, Joaquim},
   title={Variations of Nagata's conjecture},
   conference={
      title={A celebration of algebraic geometry},
   },
   book={
      series={Clay Math. Proc.},
      volume={18},
      publisher={Amer. Math. Soc., Providence, RI},
   },
   date={2013},
   pages={185--203},
   review={\MR{3114941}},
}


\bib{Co}{article}{
   author={Cox, David A.},
   title={The homogeneous coordinate ring of a toric variety},
   journal={J. Algebraic Geom.},
   volume={4},
   date={1995},
   number={1},
   pages={17--50},
   issn={1056-3911},
   review={\MR{1299003}},
}


\bib{ck}{article}{
   author={Cutkosky, Steven Dale},
   author={Kurano, Kazuhiko},
   title={Asymptotic regularity of powers of ideals of points in a weighted
   projective plane},
   journal={Kyoto J. Math.},
   volume={51},
   date={2011},
   number={1},
   pages={25--45},
   issn={2156-2261},
   review={\MR{2784746}},
   doi={10.1215/0023608X-2010-019},
}

\bib{gk}{article}{
   author={Gonz{\'a}lez, Jos{\'e} Luis},
   author={Karu, Kalle},
   title={Some non-finitely generated Cox rings},
   journal={Compos. Math.},
   volume={152},
   date={2016},
   number={5},
   pages={984--996},
   issn={0010-437X},
   review={\MR{3505645}},
   doi={10.1112/S0010437X15007745},
}

\bib{GoNiWa}{article}{
   author={Goto, Shiro},
   author={Nishida, Koji},
   author={Watanabe, Keiichi},
   title={Non-Cohen-Macaulay symbolic blow-ups for space monomial curves and
   counterexamples to Cowsik's question},
   journal={Proc. Amer. Math. Soc.},
   volume={120},
   date={1994},
   number={2},
   pages={383--392},
   issn={0002-9939},
   review={\MR{1163334}},
   doi={10.2307/2159873},
}

\bib{hkl}{article}{
   author={Hausen, J{\"u}rgen},
   author={Keicher, Simon},
   author={Laface, Antonio},
   title={Computing Cox rings},
   journal={Math. Comp.},
   volume={85},
   date={2016},
   number={297},
   pages={467--502},
   issn={0025-5718},
   review={\MR{3404458}},
   doi={10.1090/mcom/2989},
}

\bib{HuKe}{article}{
   author={Hu, Yi},
   author={Keel, Sean},
   title={Mori dream spaces and GIT},
   note={Dedicated to William Fulton on the occasion of his 60th birthday},
   journal={Michigan Math. J.},
   volume={48},
   date={2000},
   pages={331--348},
   issn={0026-2285},
   review={\MR{1786494}},
   doi={10.1307/mmj/1030132722},
}

\bib{MiSt}{book}{
   author={Miller, Ezra},
   author={Sturmfels, Bernd},
   title={Combinatorial commutative algebra},
   series={Graduate Texts in Mathematics},
   volume={227},
   publisher={Springer-Verlag, New York},
   date={2005},
   pages={xiv+417},
   isbn={0-387-22356-8},
   review={\MR{2110098}},
}


\bib{diophantic}{book}{
   author={Ram{\'{\i}}rez Alfons{\'{\i}}n, J. L.},
   title={The Diophantine Frobenius problem},
   series={Oxford Lecture Series in Mathematics and its Applications},
   volume={30},
   publisher={Oxford University Press, Oxford},
   date={2005},
   pages={xvi+243},
   isbn={978-0-19-856820-9},
   isbn={0-19-856820-7},
   review={\MR{2260521}},
   doi={10.1093/acprof:oso/9780198568209.001.0001},
}

\bib{za}{article}{
   author={Zariski, Oscar},
   title={The theorem of Riemann-Roch for high multiples of an effective
   divisor on an algebraic surface},
   journal={Ann. of Math. (2)},
   volume={76},
   date={1962},
   pages={560--615},
   issn={0003-486X},
   review={\MR{0141668}},
}

\end{biblist}
\end{bibdiv}
\end{document}